\documentclass[10pt,a4paper]{amsart}
\usepackage[utf8]{inputenc}
\usepackage[T1]{fontenc}

\pdfoutput=1

\usepackage[foot]{amsaddr}

\title[Fat minors in finitely presented groups]{Fat minors in finitely presented groups}
\author{Joseph Paul MacManus}
\date{20th August, 2024}

\address{Mathematical Institute,  University of Oxford, Oxford, OX2 6GG, UK}
\email{macmanus@maths.ox.ac.uk}

\usepackage[numbers]{natbib}

\usepackage{amsfonts}
\usepackage{amsmath}
\usepackage{amsthm}
\usepackage{amssymb}
\usepackage{stmaryrd}
\usepackage{mathrsfs}  
\usepackage{thmtools}
\usepackage{subcaption}

\usepackage{graphicx}
\usepackage[dvipsnames]{xcolor}
\usepackage{tikz}
\usepackage{tgpagella}
\usepackage[colorlinks]{hyperref} 
\hypersetup{
    linkcolor=black,
    citecolor=blue,
}
\usepackage{fancyref}
\usepackage{todonotes}

\DeclareMathOperator{\diam}{diam}
\DeclareMathOperator{\length}{length}

\DeclareMathOperator{\dHaus}{Haus}
\DeclareMathOperator{\dist}{d}
\DeclareMathOperator{\Aut}{Aut}

\newcommand{\R}{\mathbb{R}}
\newcommand{\Z}{\mathbb{Z}}
\newcommand{\N}{\mathbb{N}}

\newcommand{\into}{\hookrightarrow}

\newcommand{\actson}{\curvearrowright}

\newcommand{\pres}[2]{\langle #1 \ ; \ #2 \rangle}

\newtheorem{theorem}{Theorem}[section]

\newtheorem{alphtheorem}{Theorem}

\newtheorem{alphcor}[alphtheorem]{Corollary}

\newtheorem*{theorem*}{Theorem}

\newtheorem*{conjecture*}{Conjecture}

\newtheorem*{falseconjecture*}{False Conjecture}

\newtheorem{proposition}[theorem]{Proposition}
\newtheorem{lemma}[theorem]{Lemma}
\newtheorem{corollary}[theorem]{Corollary}

\newtheorem{conjecture}[theorem]{Conjecture}

\theoremstyle{definition}
\newtheorem{definition}[theorem]{Definition}

\newtheorem{remark}[theorem]{Remark}

\makeatletter
\newcommand{\myitem}[1]{%
\item[#1]\protected@edef\@currentlabel{#1}%
}
\makeatother

\begin{document}

\begin{abstract}
    We show that a finitely presented group virtually admits a planar Cayley graph if and only if it is asymptotically minor-excluded, partially answering a conjecture of Georgakopoulos and Papasoglu in the affirmative.     
\end{abstract}

\maketitle



\section{Introduction}

Given graphs $\Gamma$ and $H$, recall that $\Gamma$ is said to contain an $H$-\textit{minor} if $H$ can be obtained by contracting edges of a subgraph of $\Gamma$. The idea of a minor is one of the most fundamental and important concepts in all of graph theory, the highlight of their study being the famous Robertson--Seymour `Graph Minors Project' \cite{robertson1983graph}. 

In the realm of geometric group theory, there has been some recent work in the direction of understanding when the minors present in a Cayley graph affect the structure of the group. For example, an infinite graph $\Gamma$ is said to be \textit{minor-excluded} if there exists a finite graph which is not a minor of $\Gamma$, and a theorem of Khukhro \cite{khukhro2023characterisation} states that \textbf{every} Cayley graph of a finitely generated group $G$ is minor-excluded if and only if $G$ is virtually free.
In a similar vein, theorems of Esperet, Giocanti, Legrand-Duchesne and the author \cite{esperet2023structure, esperet2024coarse, macmanus2023accessibility} combine to show that if $G$ admits \textbf{some} minor-excluded Cayley graph then $G$ is virtually a free product of free and surface groups.

Since `the' Cayley graph of a finitely generated group is only well-defined up to quasi-isometry, it is likely hard to say much more than the above about minors in Cayley graphs. 
However, recently interest has grown in a new topic known as `coarse graph theory' where one studies the `large-scale' features of graphs. Popularised by Georgakopoulos and Papasoglu in their seminal paper \cite{georgakopoulos2023graph}, the analogue of a graph minor in this field is something called an `asymptotic minor'. Roughly speaking, the idea is that a graph $H$ is called an asymptotic minor of a metric space $X$ if we see arbitrarily `fat' copies of $H$ within $X$. See Section~\ref{sec:fat-minors} below for a precise definition. The benefit here is that the set of asymptotic minors of a graph is easily seen to be a quasi-isometry invariant. 
Originally, Georgakopoulos and Papasoglu posed the following conjecture.

\begin{falseconjecture*}[{\cite[1.1]{georgakopoulos2023graph}}]
	Let $X$, $H$ be connected graphs with $H$ finite. Then $H$ is not an asymptotic minor of $X$ if and only if $X$ is quasi-isometric to some graph $Y$ which contains no $H$-minor. 
\end{falseconjecture*}

The `if' direction of the above is an easy exercise. Regarding the converse, 
several positive results are known for certain small $H$ \cite{manning2005geometry, chepoi2012constant, fujiwara2023coarse, georgakopoulos2023graph}. Recently, a counterexample was found by Davies, Hickingbotham, Illingworth, and McCarty \cite{davies2024fat}. 

Georgakopoulos and Papasoglu also posed several related questions. In particular, the following conjecture is still  open.

\begin{conjecture}[{\cite[9.3]{georgakopoulos2023graph}}]\label{con:main-conjecture}
    A connected, locally finite, quasi-transitive graph $X$ is asymptotically minor-excluded if and only if it is quasi-isometric to a planar graph. 
\end{conjecture}

In this paper, we partially answer Conjecture~\ref{con:main-conjecture} in the affirmative, under the additional assumption that $X$ is a Cayley graph of some finitely presented group.

\begin{restatable}{alphtheorem}{fp}\label{thm:main-result}
A finitely \textbf{presented} group is asymptotically minor-excluded if and only if some finite index subgroup of $G$ admits a planar Cayley graph. 
\end{restatable}

Since the asymptotic minors of a graph are a quasi-isometry invariant, we immediately obtain the following corollary. 

\begin{alphcor}
    Let $G$ be a finitely presented group which is quasi-isometric to some minor-excluded graph. Then some finite-index subgroup of $G$ admits a planar Cayley graph. 
\end{alphcor}

This can be seen as a strengthening for finitely presented groups of \cite[Cor.~D]{macmanus2023accessibility}, which records the same conclusion for those finitely generated groups which are quasi-isometric to planar graphs. 

Dropping the hypothesis of finite presentability would require a different approach to that presented in this paper, as the geometry of an arbitrary finitely generated group, or more generally a quasi-transitive graph, has the potential to be far more pathological than that of any finitely presented group. For example, one such challenge comes from the existence of inaccessible groups \cite{dunwoody1993inaccessible}.

\subsection*{Outline}

In order to prove Theorem~\ref{thm:main-result}, we may first note that Dunwoody's accessibility theorem \cite{dunwoody1985accessibility} allows us to immediately restrict to the one-ended case. Then, known structure results  imply that given a one-ended finitely presented group $G$ which is not a virtual surface group, at least one of the following holds:
\begin{enumerate}
	
	\item $G$ contains a infinite descending chain of one-ended subgroups
	$$
	G \geq G_1 \geq G_2 \ldots ,
	$$ 
	where each has infinite index in the last. 
	
	\item $G$ contains an infinite-index surface subgroup. 
	\item\label{itm:third} $G$ contains a one-ended, finitely presented subgroup which is not virtually a surface group, and does not split over a two-ended subgroup.
\end{enumerate}

To prove Theorem~\ref{thm:main-result} we thus proceed by considering each of these cases individually, making use of the hypothesised geometry present to directly construct fat minors in our Cayley graph. 

The first two cases are fairly elementary and self-contained, and do not require the hypothesis of finite presentability. However, the proof in the third case heavily relies on Papasoglu's geometric characterisation of two-ended splittings \cite{papasoglu2005quasi}, which fails for arbitrary one-ended groups \cite{papasoglu2012splittings}.

\section{Preliminaries}\label{sec:prelims}

\subsection{Notation and terminology}

Given a metric space $X$, we denote by $\dist_X$ its metric. If there is no risk of ambiguity then we may write $\dist = \dist_X$. Given $x \in X$ and $R > 0$ we denote by $N_R(x)$ the closed $R$-ball around $x$. If further clarity is needed then we may instead write $B_X(x;R)$ to make the choice of ambient space clear. We define similarly notation for the closed $R$-neighbourhood of a subset of $X$.

If $A, B \subset X$ are subsets of $X$ then we write $\dist_X(A,B)$ to mean their \textit{infimal distance}. That is, we define
$$
\dist_X(A,B) := \inf \{\dist_X(a,b) : a \in A, \ b \in B\}
$$
Given $x \in X$ and $A \subset X$ we will simplify notation by writing $\dist_X(x, A) = \dist_X(\{x\}, A)$. 
We denote the \textit{Hausdorff distance} of $A$ and $B$ in $X$ as 
$$
\dHaus_X(A,B) := \inf \{R >0 : A \subset N_R(B) \text{ and } B \subset N_R(A) \}.
$$
Given $R, \varepsilon > 0$ and $Z \subset X$, we will write
$$
A_{R \pm \varepsilon}(Z) := \{x \in X : \dist_X(x, Z) \in [R-\varepsilon, R+\varepsilon] \}. 
$$
We will also write 
$$
S_{R}(Z) := \{x \in X : \dist_X(x, Z) = R \}.
$$

If $Y$ is another metric space and $\lambda \geq 1$ is a constant, then a map $\varphi : X \to Y$ is called a \textit{$\lambda$-quasi-isometric embedding} if 
$$
\tfrac 1 \lambda \dist_X(x,y) - \lambda \leq \dist_Y(\varphi(x), \varphi(y)) \leq \lambda \dist_X(x,y) + \lambda,
$$
for all $x, y \in X$.
Furthermore, if $\varphi$ satisfies that for every $y \in Y$ there exists $x \in X$ such that 
$$
\dist_Y(y, \varphi(x)) \leq \lambda,
$$
then we call $\varphi$ a $\lambda$-quasi-isometry. Given two quasi-isometries $\varphi : X \to Y$, $\psi : Y \to X$, and $\lambda \geq 1$, we say that \emph{$\psi$ is a $\lambda$-quasi-inverse to $\varphi$} if for all $x \in X$ we have that $\dist_X(\psi \circ \varphi(x), x) \leq \lambda$.  It is easy to check that every quasi-isometry has a quasi-inverse. 

We will also need the following weaker notion. We say that a map $F : X \to Y$ is a \textit{coarse embedding} if there exists $\xi_-, \xi_+ : [0,\infty) \to [0,\infty)$ such that 
$
\lim_{t \to \infty} \xi_\pm (t) = \infty
$
and 
$$
\xi_-(\dist_X(x, y)) \leq \dist_Y(F(x), F(y)) \leq \xi_+(\dist_X(x, y))
$$
for all $x, y \in X$. It is easy to see that we may assume without loss of generality that $\xi_-$ is increasing, surjective, and every value in $[0,\infty)$ except for 0 has exactly one preimage under $\xi_-$. The canonical example of a coarse embedding is that if $X$, $Y$ are Cayley graphs of finitely generated groups $H$ and $G$ respectively where $H \leq G$, then $X$ coarsely embeds into $Y$. 

Now, suppose that $X$ is a (connected, locally finite) graph. For us, a \textit{graph} shall always be taken to mean a 1-dimensional simplicial complex. We write $VX$ and $EX$ for the vertex set and edge set of $X$, respectively.  For the purposes of this paper, every graph we consider will be assumed to be simplicial. We denote by $\Delta(X)$ the maximal degree of any vertex of $X$. We will normally abuse notation and identify $X$ with its geometric realisation as a CW-complex. We then metrise $VX$ via the natural path metric and extend this metric naturally to all of $X$ by identifying each edge with a copy of the unit interval. 

Throughout this paper we will abuse notation relating to paths. Sometimes we may view path a path $p$ in a graph $X$ as subgraph, and other times we will parameterise $p$ as a map $p : I \to X$, where $I$ is some interval. In the latter case, unless otherwise stated we will usually take $I = [0, \ell]$ for some $\ell \geq 0$, and $p$ will be taken to be parameterised at unit speed. That is, $\ell = \length(p)$.

\subsection{Fat and asymptotic minors}\label{sec:fat-minors}

Let $(X,\dist)$ be a length space, $H$ a connected simplicial graph, and $K \geq 1$. Then a \textit{$K$-fat $H$-minor} is a subspace $M \subset X$ equipped with a decomposition into pieces
$$
M =  \bigcup_{v \in VH} U_v \cup \bigcup_{e \in EH} P_e,
$$
satisfying the following:
\begin{enumerate}
    \item each $U_v$ is path connected, 

    \item if $e \in EH$ is an edge with endpoints $u, v \in VH$, then $P_e$ is a path connecting $U_u$ to $U_v$, 

    \item For all distinct $A, B \in \{U_v : v \in VH\} \cup \{P_e : e \in EH\}$ we have that $\dist(A, B) > K$, unless $A = P_e$, $B= U_v$, or vice versa, and $v$ is an endpoint of $e$.

\end{enumerate}
The $U_v$ are called the \textit{branch sets} of $M$ and the $P_e$ are called the \textit{edge paths}.

If $X$ contains a $K$-fat $H$-minor for every $K \geq 1$, then we say that $H$ is an \textit{asymptotic minor} of $X$. If there exists some finite $H$ such that $H$ is not an asymptotic minor of $X$, then we say that $X$ is \textit{asymptotically minor-excluded}. 

It is immediate that if $X$ and $Y$ are length spaces and $X$ coarsely embeds into $Y$, then every asymptotic minor of $X$ is also an asymptotic minor of $Y$.

As a warm-up to constructing fat minors, we record the following easy example of groups which are not asymptotically minor-excluded. 
Given integers $n, m \neq 0$, the \textit{Baumslag--Solitar group} $BS(n,m)$ is given by the one-relator presentation
$$
BS(n,m) := \pres{x,t}{x^n = t^{-1} x^m t}.
$$
For example, $BS(1,1) \cong \Z^2$ and $BS(1,-1)$ is the fundamental group of a Klein bottle. Outside of these basic examples, we note the following. 

\begin{proposition}\label{lem:bs}
    Given $n > 1$, $m \neq 0$, we have that the group $BS(n,m)$ is not asymptotically minor-excluded.
\end{proposition}

\begin{proof}
    This follows very quickly from inspecting the standard  Cayley graphs of these groups. We leave this as an exercise to the reader. 
\end{proof}


\section{Descending chains of one-ended subgroups}\label{sec:desc-chain}

In this section we prove that groups which contain an infinite descending chain of one-ended, infinite-index subgroups cannot be asymptotically minor-excluded. 

Throughout this paper, we will use the following terminology. 

\begin{definition}[Normal geodesic]\label{def:normal}
	Let $X$ be a graph and $Y \subset X$ a subgraph. 
    Let $\rho : I \to X$ be a (finite or one-way infinite) geodesic with one endpoint $y_0$ lying on $Y$. We say that $\rho$ is a \textit{normal to $Y$} if 
    $
    \dist_X(\rho(t), Y) = t
    $
    for all $t \in I$. 
\end{definition}
 
We first note the following standard application of the Arzela--Ascoli theorem, which says that infinite normals often exist.

\begin{lemma}\label{lem:normal-ray-subgroup}
    Let $X$ be a connected, locally finite graph. Let $Y \subset X$ be a subgraph such that the setwise satibliser of $Y$ in $ \Aut(X)$ acts coboundedly on $Y$, and $\dHaus_X(Y, X) = \infty$. Then there exists a one-way infinite geodesic ray $\rho : [0,\infty) \to X$ which is normal to $Y$. 
\end{lemma}

\begin{proof}
    Since $\dHaus_X(Y, X) = \infty$ we have that for all $k > 0$ there exists a geodesic segment $\rho_k : [0,k] \to X$ of length $k$ which is normal to $Y$. Using the action of $\Gamma_Y$, we may assume that there exists some compact subgraph $K \subset Y$ such that $\rho_k(0) \in K$ for all $k \geq 0$. 

    By the Arzela--Ascoli theorem, the sequence $(\rho_k)$ contains a subsequence $(\rho_{n_k})$ which converges uniformly on compact subsets to an infinite geodesic ray $\rho$. 
    In particular, for all $N \geq 1$ there exists $M \geq 1$ such that for all $k \geq M$ we have 
    $$
    \rho|_{[0,N]} = \rho_{n_k}|_{[0,N]}.
    $$ 
    From this property, it is clear that $\rho$ is normal to $Y$.
\end{proof}

Our proof of Theorem~\ref{thm:descending-chain} will mostly follow from the next lemma.

\begin{lemma}\label{lem:induct}
    Let $G$ be finitely generated, one-ended. 
    Fix $m \geq 1$. Suppose there exists some infinite-index, finitely generated $H \leq G$ such that $H$ contains $K_m$ as an asymptotic minor. 
    Then $G$ contains $K_{m+1}$ as an asymptotic minor. 
\end{lemma}

\begin{proof}
	Let $X$ be some fixed Cayley graph of $G$. 
	Fix $K > 0$, $m \geq 0$, and suppose that $H$ contains $K_m$ as an asymptotic minor. 
    Let $Y \subset X$ be a connected tubular neighbourhood of $H$, so that $H \actson Y$ freely and cocompactly. Note that $Y$ therefore also contains $K_m$ as an asymptotic minor by assumption, and also the inclusion $Y \into X$ is a coarse embedding. 
    Fix $B \geq 1$ such that the action of $H$ on $Y$ is $B$-cobounded (with respect to the ambient metric on $X$). 
    
    Using Lemma~\ref{lem:normal-ray-subgroup}, let $\rho : [0, \infty) \to X$ be a one-way infinite geodesic ray based in $H$ such that $\dist_X(\rho(t), H) = t$ for every $t \geq 0$. 
    Since $Y$ contains $K_m$ as an asymptotic minor and the inclusion $Y \into X$ is a coarse embedding, we have that there exists arbitrarily fat $K_m$-minors in $X$ which are contained in $Y$. Let $K' = 10(K+B)$. Let $M \subset Y$ be a $K'$-fat (with respect to the ambient metric on $X$) $K_m$-minor. Let $U_1, \ldots, U_m$ be the branch sets of $M$. 

    Using the action of $H$ on $Y$, we place a copy $\rho_i$ of $\rho$ near each $U_i$, so $\dist_X(\rho_i(0), U_i) \leq B$ and  $\dist_X(\rho_i(0), U_j) > K'-B$ for $j \neq i$. 
    Since $X$ is one-ended, let 
    $
    U_{m+1}$ denote the unique unbounded component of $X \setminus N_{5K}(M). 
    $
    We have that all the $\rho_i$ are eventually contained in $U_{m+1}$. In particular, it is clear that 
    $$
    M' :=  M \cup U_{m+1} \cup \bigcup_i \rho_i,
    $$
    can be decomposed as a $K$-fat $K_{m+1}$-minor. See Figure~\ref{fig:km1} for a cartoon of this construction.
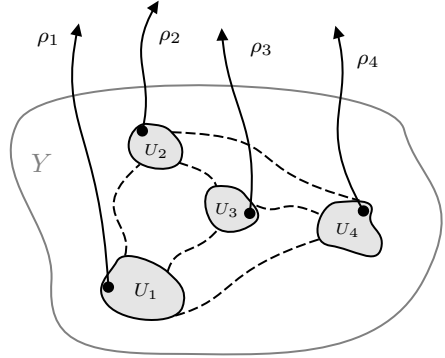
\begin{figure}
    \centering

\tikzset{every picture/.style={line width=0.75pt}} 

\begin{tikzpicture}[x=0.75pt,y=0.75pt,yscale=-1,xscale=1]

\draw  [color={rgb, 255:red, 128; green, 128; blue, 128 }  ,draw opacity=1 ] (162.6,93.4) .. controls (176.61,65.96) and (348.35,70.2) .. (355.4,97) .. controls (362.45,123.8) and (387.49,130.36) .. (365,173) .. controls (342.51,215.64) and (284.46,209.87) .. (238.1,209.87) .. controls (191.73,209.87) and (177.55,201.95) .. (179.4,173.4) .. controls (181.25,144.85) and (148.59,120.84) .. (162.6,93.4) -- cycle ;
\draw  [color={rgb, 255:red, 0; green, 0; blue, 0 }  ,draw opacity=1 ][fill={rgb, 255:red, 229; green, 229; blue, 229 }  ,fill opacity=1 ] (212.08,163.28) .. controls (215.62,156.6) and (232.76,164.24) .. (237.39,168.69) .. controls (242.02,173.15) and (248,177.29) .. (244.74,186.52) .. controls (241.47,195.75) and (215.62,194.47) .. (207.45,185.24) .. controls (199.29,176.01) and (208.54,169.97) .. (212.08,163.28) -- cycle ;
\draw  [color={rgb, 255:red, 0; green, 0; blue, 0 }  ,draw opacity=1 ][fill={rgb, 255:red, 229; green, 229; blue, 229 }  ,fill opacity=1 ] (227.05,94.22) .. controls (233.04,94.85) and (236.35,94.96) .. (239.36,98.37) .. controls (242.36,101.77) and (246.25,104.93) .. (244.13,111.98) .. controls (242.01,119.04) and (225.22,118.06) .. (219.91,111.01) .. controls (214.61,103.96) and (221.06,93.58) .. (227.05,94.22) -- cycle ;
\draw  [color={rgb, 255:red, 0; green, 0; blue, 0 }  ,draw opacity=1 ][fill={rgb, 255:red, 229; green, 229; blue, 229 }  ,fill opacity=1 ] (321.7,133.66) .. controls (329.65,133.04) and (341.63,131.45) .. (338.97,139.02) .. controls (336.31,146.6) and (346.21,147.55) .. (343.99,156.7) .. controls (341.76,165.85) and (330.47,157.23) .. (318.54,155.44) .. controls (306.62,153.64) and (313.75,134.27) .. (321.7,133.66) -- cycle ;
\draw  [color={rgb, 255:red, 0; green, 0; blue, 0 }  ,draw opacity=1 ][fill={rgb, 255:red, 229; green, 229; blue, 229 }  ,fill opacity=1 ] (261.71,125.55) .. controls (268.15,125.09) and (273.38,126.33) .. (275.99,129.95) .. controls (278.6,133.58) and (286.65,139.73) .. (280.14,144.45) .. controls (273.62,149.18) and (263.71,150.93) .. (259.1,143.42) .. controls (254.49,135.91) and (255.27,126.02) .. (261.71,125.55) -- cycle ;
\draw [color={rgb, 255:red, 0; green, 0; blue, 0 }  ,draw opacity=1 ] [dash pattern={on 3.75pt off 1.5pt}]  (216.16,160.42) .. controls (220.79,143.55) and (196.29,138.46) .. (223.51,114.58) ;
\draw [color={rgb, 255:red, 0; green, 0; blue, 0 }  ,draw opacity=1 ] [dash pattern={on 3.75pt off 1.5pt}]  (237.39,168.69) .. controls (241.2,154.69) and (251.27,160.74) .. (262.43,147.37) ;
\draw [color={rgb, 255:red, 0; green, 0; blue, 0 }  ,draw opacity=1 ] [dash pattern={on 3.75pt off 1.5pt}]  (241.2,190.65) .. controls (271.96,182.06) and (262.97,168.37) .. (313.32,152.14) ;
\draw [color={rgb, 255:red, 0; green, 0; blue, 0 }  ,draw opacity=1 ] [dash pattern={on 3.75pt off 1.5pt}]  (239.36,98.37) .. controls (301.89,100.58) and (270.05,116.18) .. (333.19,132.73) ;
\draw [color={rgb, 255:red, 0; green, 0; blue, 0 }  ,draw opacity=1 ] [dash pattern={on 3.75pt off 1.5pt}]  (281.75,135.27) .. controls (295.36,141.32) and (292.37,130.18) .. (312.78,139.41) ;
\draw [color={rgb, 255:red, 0; green, 0; blue, 0 }  ,draw opacity=1 ] [dash pattern={on 3.75pt off 1.5pt}]  (242.83,113.63) .. controls (249.09,115.86) and (259.44,118.09) .. (261.71,125.55) ;
\draw [color={rgb, 255:red, 0; green, 0; blue, 0 }  ,draw opacity=1 ]   (224.77,97.74) .. controls (231.69,68.1) and (215.42,61.73) .. (231.6,34.35) ;
\draw [shift={(232.91,32.2)}, rotate = 122.11] [fill={rgb, 255:red, 0; green, 0; blue, 0 }  ,fill opacity=1 ][line width=0.08]  [draw opacity=0] (6.25,-3) -- (0,0) -- (6.25,3) -- cycle    ;
\draw [shift={(224.77,97.74)}, rotate = 283.13] [color={rgb, 255:red, 0; green, 0; blue, 0 }  ,draw opacity=1 ][fill={rgb, 255:red, 0; green, 0; blue, 0 }  ,fill opacity=1 ][line width=0.75]      (0, 0) circle [x radius= 2.34, y radius= 2.34]   ;
\draw [color={rgb, 255:red, 0; green, 0; blue, 0 }  ,draw opacity=1 ]   (207.82,176.31) .. controls (203.56,96.97) and (182.01,95.4) .. (192.61,45.94) ;
\draw [shift={(193.12,43.65)}, rotate = 102.81] [fill={rgb, 255:red, 0; green, 0; blue, 0 }  ,fill opacity=1 ][line width=0.08]  [draw opacity=0] (6.25,-3) -- (0,0) -- (6.25,3) -- cycle    ;
\draw [shift={(207.82,176.31)}, rotate = 266.93] [color={rgb, 255:red, 0; green, 0; blue, 0 }  ,draw opacity=1 ][fill={rgb, 255:red, 0; green, 0; blue, 0 }  ,fill opacity=1 ][line width=0.75]      (0, 0) circle [x radius= 2.34, y radius= 2.34]   ;
\draw [color={rgb, 255:red, 0; green, 0; blue, 0 }  ,draw opacity=1 ]   (278.4,139.2) .. controls (283.66,93.54) and (267.48,87.32) .. (264.55,48.44) ;
\draw [shift={(264.39,46.02)}, rotate = 86.62] [fill={rgb, 255:red, 0; green, 0; blue, 0 }  ,fill opacity=1 ][line width=0.08]  [draw opacity=0] (6.25,-3) -- (0,0) -- (6.25,3) -- cycle    ;
\draw [shift={(278.4,139.2)}, rotate = 276.57] [color={rgb, 255:red, 0; green, 0; blue, 0 }  ,draw opacity=1 ][fill={rgb, 255:red, 0; green, 0; blue, 0 }  ,fill opacity=1 ][line width=0.75]      (0, 0) circle [x radius= 2.34, y radius= 2.34]   ;
\draw [color={rgb, 255:red, 0; green, 0; blue, 0 }  ,draw opacity=1 ]   (335.15,138.01) .. controls (309.21,91.58) and (332.28,85.72) .. (321.82,47.23) ;
\draw [shift={(321.13,44.83)}, rotate = 73.42] [fill={rgb, 255:red, 0; green, 0; blue, 0 }  ,fill opacity=1 ][line width=0.08]  [draw opacity=0] (6.25,-3) -- (0,0) -- (6.25,3) -- cycle    ;
\draw [shift={(335.15,138.01)}, rotate = 240.81] [color={rgb, 255:red, 0; green, 0; blue, 0 }  ,draw opacity=1 ][fill={rgb, 255:red, 0; green, 0; blue, 0 }  ,fill opacity=1 ][line width=0.75]      (0, 0) circle [x radius= 2.34, y radius= 2.34]   ;

\draw (218.99,173.09) node [anchor=north west][inner sep=0.75pt]  [font=\scriptsize,color={rgb, 255:red, 0; green, 0; blue, 0 }  ,opacity=1 ]  {$U_{1}$};
\draw (224.27,102.7) node [anchor=north west][inner sep=0.75pt]  [font=\tiny,color={rgb, 255:red, 0; green, 0; blue, 0 }  ,opacity=1 ]  {$U_{2}$};
\draw (319.75,142.48) node [anchor=north west][inner sep=0.75pt]  [font=\tiny,color={rgb, 255:red, 0; green, 0; blue, 0 }  ,opacity=1 ]  {$U_{4}$};
\draw (259.54,132.53) node [anchor=north west][inner sep=0.75pt]  [font=\tiny,color={rgb, 255:red, 0; green, 0; blue, 0 }  ,opacity=1 ]  {$U_{3}$};
\draw (167.81,107.82) node [anchor=north west][inner sep=0.75pt]  [color={rgb, 255:red, 128; green, 128; blue, 128 }  ,opacity=1 ]  {$Y$};
\draw (170.77,47.16) node [anchor=north west][inner sep=0.75pt]  [font=\footnotesize,color={rgb, 255:red, 0; green, 0; blue, 0 }  ,opacity=1 ]  {$\rho _{1}$};
\draw (231.57,45.96) node [anchor=north west][inner sep=0.75pt]  [font=\footnotesize,color={rgb, 255:red, 0; green, 0; blue, 0 }  ,opacity=1 ]  {$\rho _{2}$};
\draw (277.17,51.96) node [anchor=north west][inner sep=0.75pt]  [font=\footnotesize,color={rgb, 255:red, 0; green, 0; blue, 0 }  ,opacity=1 ]  {$\rho _{3}$};
\draw (330.37,57.96) node [anchor=north west][inner sep=0.75pt]  [font=\footnotesize,color={rgb, 255:red, 0; green, 0; blue, 0 }  ,opacity=1 ]  {$\rho _{4}$};

\end{tikzpicture}

    \caption{Extending a fat $K_{m}$-minor in $Y$ to a fat $K_{m+1}$-minor in $X$.}
    \label{fig:km1}
\end{figure}
\end{proof}

From Lemma~\ref{lem:induct}, we may easily deduce the following.

\begin{theorem}\label{thm:descending-chain}
Let $G$ be a finitely generated group. Suppose  there exists an infinite descending chain of subgroups
        $$
        G \geq G_0 \geq G_1 \geq  \ldots \geq G_i \geq \ldots 
        $$
        such that for all $i \geq 0$ we have that $G_i$ is one-ended and $|G_i : G_{i+1}| = \infty$.
    Then $G$ is not asymptotically minor-excluded. 
\end{theorem}

\begin{proof}
    Given $n \geq 0$ we have that $G_n$ trivially contains $K_0$ as an asymptotic minor. Applying Lemma~\ref{lem:induct} a total of $n$ times, we deduce that $G$ contains $K_n$ as an asymptotic minor. Since $n$ was arbitrary, the result follows. 
\end{proof}

\section{Groups with surface subgroups}\label{sec:plane}

In this section we prove that a one-ended, finitely generated group with an infinite-index surface subgroup is not asymptotically minor-excluded. 

\subsection{Idea of proof}

Given any finite graph $F$, we can consider a particular drawing of this graph, where we allow edges to cross each other. We can then view this drawing itself as a planar graph $P$, where the `crossings' correspond to degree-four vertices, which we will call \textit{marked} vertices.
Given a one-ended finitely generated group $G$ with Cayley graph $X$ and an infinite-index surface subgroup $H \leq G$, our strategy to build fat minors of an arbitrary $F$ is thus as follows:
\begin{enumerate}
	\item Consider a particular `drawing' of $F$ and its corresponding planar graph $P$. 
	\item Draw a very large, very controlled copy of $P$ in the (hyperbolic or Euclidean) plane. 
	\item Embed this copy of $P$ into a neighbourhood of $H$ in $X$, and use the aforementioned control together with the action of $H$ to turn the marked vertices back into `fat crossings'. 
\end{enumerate}
See Figure~\ref{fig:surface-final} for an illustration of this strategy.

\subsection{Controlled fat minors in the plane}

It is obvious that the (hyperbolic or Euclidean) plane contains every finite planar graph as an asymptotic minor. Our immediate goal is to gain a bit more control over the drawings of these fat minors. 

We first state the following standard definition. 

\begin{definition}[Coarse cover]
    Let $X$ be a metric space, $\varepsilon > 0$. Then an \textit{$\varepsilon$-cover} of $X$ is a subset $N \subset X$ such that for all $x \in X$ there exists $a \in N$ such that $\dist(a,x) \leq \varepsilon$. 

    If $N$ is an $\varepsilon$-cover for some $\varepsilon > 0$, we may suppress $\varepsilon$ from our notation and simply call $N$ a \textit{coarse cover} of $X$. 
\end{definition}

We now have the following key lemma, dubbed the 'carefully-drawn lemma'. 

\begin{lemma}[Carefully-drawn lemma]\label{lem:dial}
     Consider the (Euclidean or hyperbolic) plane. Let $N$ be a coarse cover of the plane. Let $P$ be a finite simplicial planar graph such that $\Delta(P) \leq 4$, equipped with a fixed drawing in the plane. Let $r, K > 0$ such that $K > 10r$. 
     
     Then there exists a $K$-fat $P$-minor $M$ in the plane with the following properties:
     \begin{enumerate}

         \item\label{itm:disc-branchsets} For every vertex $v \in VP$, the branch set $U_v$ is a solid circular disk of radius $r$ based at a point in $N$. 

		\item Suppose that for every $x \in N$ we are given a list of four evenly spaced angles $\theta_1, \ldots, \theta_4$. If $U_v$ is a branch set centred at $x$ then we may assume that any edge path terminating at $U_v$ does so at one of the $\theta_i$.  

        \item \label{itm:same-drawing}The cyclic order in which the edge paths leave the branch sets agrees with the cyclic  order with which the edges leave the vertices in the given drawing of $P$. 
     \end{enumerate}
\end{lemma}

\begin{remark}
	Note that we only consider the case where $\Delta(P) \leq 4$ for simplicity. This lemma should be true without this restriction, but a proof in this generality would be a bit more involved and is not necessary for our purposes. 
\end{remark}

Intuitively, Lemma~\ref{lem:dial} says that we may build a fat $P$-minor $M$ for any finite planar graph $P$, where the vertices of $P$ are circular `dials' centered at prescribed points in the plane which can be rotated without destroying the fatness of $M$. 
Condition~(\ref{itm:same-drawing}) essentially says that the constructed drawing of $M$ and the given drawing of $P$ `look the same'.

We encourage the reader to try to convince themselves that this lemma is intuitively true. For the sake of completeness we now sketch a proof of Lemma~\ref{lem:dial}, though the exact details of this construction are not particularly enlightening. 

\begin{proof}[Sketch of Lemma~\ref{lem:dial}]
	Let $N$ be an $\varepsilon$-coarse cover of the plane, and fix $r, K > 0$ such that $K > 10r$. We will build a $K$-fat $P$-minor in the plane such that the branch sets are uniform disks of radius $r$ centred at points in $N$, with the property that the edge paths meet these disks at prescribed angles. 
	
	We first note that both the Euclidean and hyperbolic planes contain arbitrarily `fat' copies of the square half-grid. Indeed, in the Euclidean plane this is a wholly trivial observation, and in the hyperbolic plane we can construct examples easily using the Poincar\'e half-plane model. See Figure~\ref{fig:half-plane}. 
	
	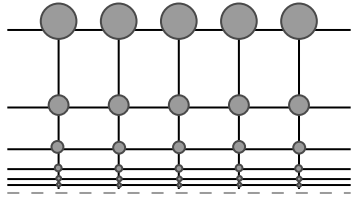
\begin{figure}[t]
		\centering

\tikzset{every picture/.style={line width=0.75pt}} 

\begin{tikzpicture}[x=0.75pt,y=0.75pt,yscale=-1,xscale=1]

\draw [color={rgb, 255:red, 155; green, 155; blue, 155 }  ,draw opacity=1 ] [dash pattern={on 4.5pt off 4.5pt}]  (74.4,207.6) -- (245.8,207.6) ;
\draw [color={rgb, 255:red, 0; green, 0; blue, 0 }  ,draw opacity=1 ]   (74.4,125.6) -- (245.8,125.6) ;
\draw [color={rgb, 255:red, 0; green, 0; blue, 0 }  ,draw opacity=1 ]   (74.4,164.6) -- (245.8,164.6) ;
\draw [color={rgb, 255:red, 0; green, 0; blue, 0 }  ,draw opacity=1 ]   (74.4,185.6) -- (245.8,185.6) ;
\draw [color={rgb, 255:red, 0; green, 0; blue, 0 }  ,draw opacity=1 ]   (74.4,195.6) -- (245.8,195.6) ;
\draw [color={rgb, 255:red, 0; green, 0; blue, 0 }  ,draw opacity=1 ]   (74.4,200.6) -- (245.8,200.6) ;
\draw [color={rgb, 255:red, 0; green, 0; blue, 0 }  ,draw opacity=1 ]   (74.4,203.6) -- (245.8,203.6) ;
\draw    (100,125.2) -- (100,205.6) ;
\draw    (130,125.2) -- (130,205.6) ;
\draw    (160,125.2) -- (160,205.6) ;
\draw    (190,125.2) -- (190,205.6) ;
\draw    (220,125.2) -- (220,205.6) ;
\draw  [color={rgb, 255:red, 74; green, 74; blue, 74 }  ,draw opacity=1 ][fill={rgb, 255:red, 155; green, 155; blue, 155 }  ,fill opacity=1 ] (91.1,121.2) .. controls (91.1,116.28) and (95.08,112.3) .. (100,112.3) .. controls (104.92,112.3) and (108.9,116.28) .. (108.9,121.2) .. controls (108.9,126.12) and (104.92,130.1) .. (100,130.1) .. controls (95.08,130.1) and (91.1,126.12) .. (91.1,121.2) -- cycle ;
\draw  [color={rgb, 255:red, 74; green, 74; blue, 74 }  ,draw opacity=1 ][fill={rgb, 255:red, 155; green, 155; blue, 155 }  ,fill opacity=1 ] (121.1,121.2) .. controls (121.1,116.28) and (125.08,112.3) .. (130,112.3) .. controls (134.92,112.3) and (138.9,116.28) .. (138.9,121.2) .. controls (138.9,126.12) and (134.92,130.1) .. (130,130.1) .. controls (125.08,130.1) and (121.1,126.12) .. (121.1,121.2) -- cycle ;
\draw  [color={rgb, 255:red, 74; green, 74; blue, 74 }  ,draw opacity=1 ][fill={rgb, 255:red, 155; green, 155; blue, 155 }  ,fill opacity=1 ] (151.1,121.2) .. controls (151.1,116.28) and (155.08,112.3) .. (160,112.3) .. controls (164.92,112.3) and (168.9,116.28) .. (168.9,121.2) .. controls (168.9,126.12) and (164.92,130.1) .. (160,130.1) .. controls (155.08,130.1) and (151.1,126.12) .. (151.1,121.2) -- cycle ;
\draw  [color={rgb, 255:red, 74; green, 74; blue, 74 }  ,draw opacity=1 ][fill={rgb, 255:red, 155; green, 155; blue, 155 }  ,fill opacity=1 ] (181.1,121.2) .. controls (181.1,116.28) and (185.08,112.3) .. (190,112.3) .. controls (194.92,112.3) and (198.9,116.28) .. (198.9,121.2) .. controls (198.9,126.12) and (194.92,130.1) .. (190,130.1) .. controls (185.08,130.1) and (181.1,126.12) .. (181.1,121.2) -- cycle ;
\draw  [color={rgb, 255:red, 74; green, 74; blue, 74 }  ,draw opacity=1 ][fill={rgb, 255:red, 155; green, 155; blue, 155 }  ,fill opacity=1 ] (211.1,121.2) .. controls (211.1,116.28) and (215.08,112.3) .. (220,112.3) .. controls (224.92,112.3) and (228.9,116.28) .. (228.9,121.2) .. controls (228.9,126.12) and (224.92,130.1) .. (220,130.1) .. controls (215.08,130.1) and (211.1,126.12) .. (211.1,121.2) -- cycle ;
\draw  [color={rgb, 255:red, 74; green, 74; blue, 74 }  ,draw opacity=1 ][fill={rgb, 255:red, 155; green, 155; blue, 155 }  ,fill opacity=1 ] (95.05,163.4) .. controls (95.05,160.67) and (97.27,158.45) .. (100,158.45) .. controls (102.73,158.45) and (104.95,160.67) .. (104.95,163.4) .. controls (104.95,166.13) and (102.73,168.35) .. (100,168.35) .. controls (97.27,168.35) and (95.05,166.13) .. (95.05,163.4) -- cycle ;
\draw  [color={rgb, 255:red, 74; green, 74; blue, 74 }  ,draw opacity=1 ][fill={rgb, 255:red, 155; green, 155; blue, 155 }  ,fill opacity=1 ] (125.05,163.4) .. controls (125.05,160.67) and (127.27,158.45) .. (130,158.45) .. controls (132.73,158.45) and (134.95,160.67) .. (134.95,163.4) .. controls (134.95,166.13) and (132.73,168.35) .. (130,168.35) .. controls (127.27,168.35) and (125.05,166.13) .. (125.05,163.4) -- cycle ;
\draw  [color={rgb, 255:red, 74; green, 74; blue, 74 }  ,draw opacity=1 ][fill={rgb, 255:red, 155; green, 155; blue, 155 }  ,fill opacity=1 ] (155.05,163.4) .. controls (155.05,160.67) and (157.27,158.45) .. (160,158.45) .. controls (162.73,158.45) and (164.95,160.67) .. (164.95,163.4) .. controls (164.95,166.13) and (162.73,168.35) .. (160,168.35) .. controls (157.27,168.35) and (155.05,166.13) .. (155.05,163.4) -- cycle ;
\draw  [color={rgb, 255:red, 74; green, 74; blue, 74 }  ,draw opacity=1 ][fill={rgb, 255:red, 155; green, 155; blue, 155 }  ,fill opacity=1 ] (185.05,163.4) .. controls (185.05,160.67) and (187.27,158.45) .. (190,158.45) .. controls (192.73,158.45) and (194.95,160.67) .. (194.95,163.4) .. controls (194.95,166.13) and (192.73,168.35) .. (190,168.35) .. controls (187.27,168.35) and (185.05,166.13) .. (185.05,163.4) -- cycle ;
\draw  [color={rgb, 255:red, 74; green, 74; blue, 74 }  ,draw opacity=1 ][fill={rgb, 255:red, 155; green, 155; blue, 155 }  ,fill opacity=1 ] (215.05,163.4) .. controls (215.05,160.67) and (217.27,158.45) .. (220,158.45) .. controls (222.73,158.45) and (224.95,160.67) .. (224.95,163.4) .. controls (224.95,166.13) and (222.73,168.35) .. (220,168.35) .. controls (217.27,168.35) and (215.05,166.13) .. (215.05,163.4) -- cycle ;
\draw  [color={rgb, 255:red, 74; green, 74; blue, 74 }  ,draw opacity=1 ][fill={rgb, 255:red, 155; green, 155; blue, 155 }  ,fill opacity=1 ] (96.45,184.43) .. controls (96.45,182.78) and (97.78,181.45) .. (99.43,181.45) .. controls (101.07,181.45) and (102.4,182.78) .. (102.4,184.43) .. controls (102.4,186.07) and (101.07,187.4) .. (99.43,187.4) .. controls (97.78,187.4) and (96.45,186.07) .. (96.45,184.43) -- cycle ;
\draw  [color={rgb, 255:red, 74; green, 74; blue, 74 }  ,draw opacity=1 ][fill={rgb, 255:red, 155; green, 155; blue, 155 }  ,fill opacity=1 ] (127.25,184.83) .. controls (127.25,183.18) and (128.58,181.85) .. (130.23,181.85) .. controls (131.87,181.85) and (133.2,183.18) .. (133.2,184.83) .. controls (133.2,186.47) and (131.87,187.8) .. (130.23,187.8) .. controls (128.58,187.8) and (127.25,186.47) .. (127.25,184.83) -- cycle ;
\draw  [color={rgb, 255:red, 74; green, 74; blue, 74 }  ,draw opacity=1 ][fill={rgb, 255:red, 155; green, 155; blue, 155 }  ,fill opacity=1 ] (157.13,184.6) .. controls (157.13,182.96) and (158.46,181.63) .. (160.1,181.63) .. controls (161.74,181.63) and (163.08,182.96) .. (163.08,184.6) .. controls (163.08,186.24) and (161.74,187.58) .. (160.1,187.58) .. controls (158.46,187.58) and (157.13,186.24) .. (157.13,184.6) -- cycle ;
\draw  [color={rgb, 255:red, 74; green, 74; blue, 74 }  ,draw opacity=1 ][fill={rgb, 255:red, 155; green, 155; blue, 155 }  ,fill opacity=1 ] (187.13,184.4) .. controls (187.13,182.76) and (188.46,181.43) .. (190.1,181.43) .. controls (191.74,181.43) and (193.08,182.76) .. (193.08,184.4) .. controls (193.08,186.04) and (191.74,187.38) .. (190.1,187.38) .. controls (188.46,187.38) and (187.13,186.04) .. (187.13,184.4) -- cycle ;
\draw  [color={rgb, 255:red, 74; green, 74; blue, 74 }  ,draw opacity=1 ][fill={rgb, 255:red, 155; green, 155; blue, 155 }  ,fill opacity=1 ] (217.13,184.8) .. controls (217.13,183.16) and (218.46,181.83) .. (220.1,181.83) .. controls (221.74,181.83) and (223.08,183.16) .. (223.08,184.8) .. controls (223.08,186.44) and (221.74,187.78) .. (220.1,187.78) .. controls (218.46,187.78) and (217.13,186.44) .. (217.13,184.8) -- cycle ;
\draw  [color={rgb, 255:red, 74; green, 74; blue, 74 }  ,draw opacity=1 ][fill={rgb, 255:red, 155; green, 155; blue, 155 }  ,fill opacity=1 ] (98.13,195.01) .. controls (98.13,194.08) and (98.88,193.33) .. (99.81,193.33) .. controls (100.74,193.33) and (101.5,194.08) .. (101.5,195.01) .. controls (101.5,195.94) and (100.74,196.7) .. (99.81,196.7) .. controls (98.88,196.7) and (98.13,195.94) .. (98.13,195.01) -- cycle ;
\draw  [color={rgb, 255:red, 74; green, 74; blue, 74 }  ,draw opacity=1 ][fill={rgb, 255:red, 155; green, 155; blue, 155 }  ,fill opacity=1 ] (98.88,200.26) .. controls (98.88,199.61) and (99.41,199.08) .. (100.06,199.08) .. controls (100.72,199.08) and (101.25,199.61) .. (101.25,200.26) .. controls (101.25,200.92) and (100.72,201.45) .. (100.06,201.45) .. controls (99.41,201.45) and (98.88,200.92) .. (98.88,200.26) -- cycle ;
\draw  [color={rgb, 255:red, 74; green, 74; blue, 74 }  ,draw opacity=1 ][fill={rgb, 255:red, 155; green, 155; blue, 155 }  ,fill opacity=1 ] (99.13,203.39) .. controls (99.13,202.87) and (99.54,202.45) .. (100.06,202.45) .. controls (100.58,202.45) and (101,202.87) .. (101,203.39) .. controls (101,203.91) and (100.58,204.33) .. (100.06,204.33) .. controls (99.54,204.33) and (99.13,203.91) .. (99.13,203.39) -- cycle ;
\draw  [color={rgb, 255:red, 74; green, 74; blue, 74 }  ,draw opacity=1 ][fill={rgb, 255:red, 155; green, 155; blue, 155 }  ,fill opacity=1 ] (128.38,195.26) .. controls (128.38,194.33) and (129.13,193.58) .. (130.06,193.58) .. controls (130.99,193.58) and (131.75,194.33) .. (131.75,195.26) .. controls (131.75,196.19) and (130.99,196.95) .. (130.06,196.95) .. controls (129.13,196.95) and (128.38,196.19) .. (128.38,195.26) -- cycle ;
\draw  [color={rgb, 255:red, 74; green, 74; blue, 74 }  ,draw opacity=1 ][fill={rgb, 255:red, 155; green, 155; blue, 155 }  ,fill opacity=1 ] (129.13,200.51) .. controls (129.13,199.86) and (129.66,199.33) .. (130.31,199.33) .. controls (130.97,199.33) and (131.5,199.86) .. (131.5,200.51) .. controls (131.5,201.17) and (130.97,201.7) .. (130.31,201.7) .. controls (129.66,201.7) and (129.13,201.17) .. (129.13,200.51) -- cycle ;
\draw  [color={rgb, 255:red, 74; green, 74; blue, 74 }  ,draw opacity=1 ][fill={rgb, 255:red, 155; green, 155; blue, 155 }  ,fill opacity=1 ] (129.38,203.64) .. controls (129.38,203.12) and (129.79,202.7) .. (130.31,202.7) .. controls (130.83,202.7) and (131.25,203.12) .. (131.25,203.64) .. controls (131.25,204.16) and (130.83,204.58) .. (130.31,204.58) .. controls (129.79,204.58) and (129.38,204.16) .. (129.38,203.64) -- cycle ;
\draw  [color={rgb, 255:red, 74; green, 74; blue, 74 }  ,draw opacity=1 ][fill={rgb, 255:red, 155; green, 155; blue, 155 }  ,fill opacity=1 ] (158.38,195.26) .. controls (158.38,194.33) and (159.13,193.58) .. (160.06,193.58) .. controls (160.99,193.58) and (161.75,194.33) .. (161.75,195.26) .. controls (161.75,196.19) and (160.99,196.95) .. (160.06,196.95) .. controls (159.13,196.95) and (158.38,196.19) .. (158.38,195.26) -- cycle ;
\draw  [color={rgb, 255:red, 74; green, 74; blue, 74 }  ,draw opacity=1 ][fill={rgb, 255:red, 155; green, 155; blue, 155 }  ,fill opacity=1 ] (159.13,200.51) .. controls (159.13,199.86) and (159.66,199.33) .. (160.31,199.33) .. controls (160.97,199.33) and (161.5,199.86) .. (161.5,200.51) .. controls (161.5,201.17) and (160.97,201.7) .. (160.31,201.7) .. controls (159.66,201.7) and (159.13,201.17) .. (159.13,200.51) -- cycle ;
\draw  [color={rgb, 255:red, 74; green, 74; blue, 74 }  ,draw opacity=1 ][fill={rgb, 255:red, 155; green, 155; blue, 155 }  ,fill opacity=1 ] (159.38,203.64) .. controls (159.38,203.12) and (159.79,202.7) .. (160.31,202.7) .. controls (160.83,202.7) and (161.25,203.12) .. (161.25,203.64) .. controls (161.25,204.16) and (160.83,204.58) .. (160.31,204.58) .. controls (159.79,204.58) and (159.38,204.16) .. (159.38,203.64) -- cycle ;
\draw  [color={rgb, 255:red, 74; green, 74; blue, 74 }  ,draw opacity=1 ][fill={rgb, 255:red, 155; green, 155; blue, 155 }  ,fill opacity=1 ] (188.38,195.01) .. controls (188.38,194.08) and (189.13,193.33) .. (190.06,193.33) .. controls (190.99,193.33) and (191.75,194.08) .. (191.75,195.01) .. controls (191.75,195.94) and (190.99,196.7) .. (190.06,196.7) .. controls (189.13,196.7) and (188.38,195.94) .. (188.38,195.01) -- cycle ;
\draw  [color={rgb, 255:red, 74; green, 74; blue, 74 }  ,draw opacity=1 ][fill={rgb, 255:red, 155; green, 155; blue, 155 }  ,fill opacity=1 ] (189.13,200.26) .. controls (189.13,199.61) and (189.66,199.08) .. (190.31,199.08) .. controls (190.97,199.08) and (191.5,199.61) .. (191.5,200.26) .. controls (191.5,200.92) and (190.97,201.45) .. (190.31,201.45) .. controls (189.66,201.45) and (189.13,200.92) .. (189.13,200.26) -- cycle ;
\draw  [color={rgb, 255:red, 74; green, 74; blue, 74 }  ,draw opacity=1 ][fill={rgb, 255:red, 155; green, 155; blue, 155 }  ,fill opacity=1 ] (189.38,203.39) .. controls (189.38,202.87) and (189.79,202.45) .. (190.31,202.45) .. controls (190.83,202.45) and (191.25,202.87) .. (191.25,203.39) .. controls (191.25,203.91) and (190.83,204.33) .. (190.31,204.33) .. controls (189.79,204.33) and (189.38,203.91) .. (189.38,203.39) -- cycle ;
\draw  [color={rgb, 255:red, 74; green, 74; blue, 74 }  ,draw opacity=1 ][fill={rgb, 255:red, 155; green, 155; blue, 155 }  ,fill opacity=1 ] (218.13,195.26) .. controls (218.13,194.33) and (218.88,193.58) .. (219.81,193.58) .. controls (220.74,193.58) and (221.5,194.33) .. (221.5,195.26) .. controls (221.5,196.19) and (220.74,196.95) .. (219.81,196.95) .. controls (218.88,196.95) and (218.13,196.19) .. (218.13,195.26) -- cycle ;
\draw  [color={rgb, 255:red, 74; green, 74; blue, 74 }  ,draw opacity=1 ][fill={rgb, 255:red, 155; green, 155; blue, 155 }  ,fill opacity=1 ] (218.88,200.51) .. controls (218.88,199.86) and (219.41,199.33) .. (220.06,199.33) .. controls (220.72,199.33) and (221.25,199.86) .. (221.25,200.51) .. controls (221.25,201.17) and (220.72,201.7) .. (220.06,201.7) .. controls (219.41,201.7) and (218.88,201.17) .. (218.88,200.51) -- cycle ;
\draw  [color={rgb, 255:red, 74; green, 74; blue, 74 }  ,draw opacity=1 ][fill={rgb, 255:red, 155; green, 155; blue, 155 }  ,fill opacity=1 ] (219.13,203.64) .. controls (219.13,203.12) and (219.54,202.7) .. (220.06,202.7) .. controls (220.58,202.7) and (221,203.12) .. (221,203.64) .. controls (221,204.16) and (220.58,204.58) .. (220.06,204.58) .. controls (219.54,204.58) and (219.13,204.16) .. (219.13,203.64) -- cycle ;

\end{tikzpicture}
		\caption{Example of a fat half-grid in the hyperbolic plane with a uniform disk at every lattice point, shown here in the Poincar\'e half-plane model. }
		\label{fig:half-plane}
	\end{figure}

It is clear that any finite planar graph $P$ satisfying $\Delta(P) \leq 4$ embeds topologically into the half-grid, in such a way that will `agree' with any given drawing of $P$.  We thus obtain arbitrarily fat $P$-minors in the plane where the branch sets are uniform solid disks of any given size and the edge-paths meet these disks are right angles. All that remains is to modify these disks so that they are centred on elements of a coarse cover, and the edge-paths meet these disks are prescribed angles. 
This can be achieved as follows. We choose our half-grid to be such that the branch-set disks have radius $R := \varepsilon+r + 5K + 1$. If $a$ lies at the centre of some branch set of the constructed fat $P$-minor, then there is some $b \in N$ which lies at most $\varepsilon$ away from $a$. It is now easy to construct four pairwise distance-$K$ paths inside the disk of radius $R$ at $a$ which terminate at on the circle of radius $r$ based at $b$ at prescribed angles, while preserving the cyclic order these paths arrive in. One such construction is shown in Figure~\ref{fig:circles}. Note that this figure depicts the construction in the Euclidean plane, but it is easy to check that an analogous construction is possible in the hyperbolic plane. 
\end{proof}
	
	\begin{figure}[t]
\centering

\tikzset{every picture/.style={line width=0.75pt}} 

\begin{tikzpicture}[x=0.75pt,y=0.75pt,yscale=-1,xscale=1]

\draw  [fill={rgb, 255:red, 247; green, 247; blue, 247 }  ,fill opacity=1 ] (130.86,157.07) .. controls (114.43,102.52) and (145.34,44.98) .. (199.89,28.55) .. controls (254.43,12.13) and (311.97,43.03) .. (328.4,97.58) .. controls (344.83,152.13) and (313.92,209.66) .. (259.38,226.09) .. controls (204.83,242.52) and (147.29,211.62) .. (130.86,157.07) -- cycle ;
\draw [color={rgb, 255:red, 128; green, 128; blue, 128 }  ,draw opacity=1 ] [dash pattern={on 1.5pt off 3pt}]  (199.89,28.22) -- (229.63,126.99) ;
\draw [color={rgb, 255:red, 128; green, 128; blue, 128 }  ,draw opacity=1 ] [dash pattern={on 1.5pt off 3pt}]  (229.63,126.99) -- (259.38,225.75) ;
\draw [color={rgb, 255:red, 128; green, 128; blue, 128 }  ,draw opacity=1 ] [dash pattern={on 1.5pt off 3pt}]  (229.63,126.99) -- (130.86,156.73) ;
\draw [color={rgb, 255:red, 128; green, 128; blue, 128 }  ,draw opacity=1 ] [dash pattern={on 1.5pt off 3pt}]  (328.4,97.24) -- (229.63,126.99) ;

\draw    (135.97,25.01) .. controls (168.03,2.21) and (178.47,25.84) .. (199.89,28.55) ;
\draw [shift={(199.89,28.55)}, rotate = 7.21] [color={rgb, 255:red, 0; green, 0; blue, 0 }  ][fill={rgb, 255:red, 0; green, 0; blue, 0 }  ][line width=0.75]      (0, 0) circle [x radius= 2.01, y radius= 2.01]   ;
\draw [shift={(133.98,26.46)}, rotate = 323.13] [color={rgb, 255:red, 0; green, 0; blue, 0 }  ][line width=0.75]    (6.56,-1.97) .. controls (4.17,-0.84) and (1.99,-0.18) .. (0,0) .. controls (1.99,0.18) and (4.17,0.84) .. (6.56,1.97)   ;
\draw    (67.19,156.54) .. controls (97.81,181.05) and (109.78,154.36) .. (131.2,157.07) ;
\draw [shift={(131.2,157.07)}, rotate = 7.21] [color={rgb, 255:red, 0; green, 0; blue, 0 }  ][fill={rgb, 255:red, 0; green, 0; blue, 0 }  ][line width=0.75]      (0, 0) circle [x radius= 2.01, y radius= 2.01]   ;
\draw [shift={(65.29,154.98)}, rotate = 40.31] [color={rgb, 255:red, 0; green, 0; blue, 0 }  ][line width=0.75]    (6.56,-1.97) .. controls (4.17,-0.84) and (1.99,-0.18) .. (0,0) .. controls (1.99,0.18) and (4.17,0.84) .. (6.56,1.97)   ;
\draw    (350.65,217.47) .. controls (309.66,203.36) and (314.07,225.87) .. (259.71,226.09) ;
\draw [shift={(259.71,226.09)}, rotate = 179.77] [color={rgb, 255:red, 0; green, 0; blue, 0 }  ][fill={rgb, 255:red, 0; green, 0; blue, 0 }  ][line width=0.75]      (0, 0) circle [x radius= 2.01, y radius= 2.01]   ;
\draw [shift={(352.56,218.14)}, rotate = 199.77] [color={rgb, 255:red, 0; green, 0; blue, 0 }  ][line width=0.75]    (6.56,-1.97) .. controls (4.17,-0.84) and (1.99,-0.18) .. (0,0) .. controls (1.99,0.18) and (4.17,0.84) .. (6.56,1.97)   ;
\draw [line width=0.75]    (362.64,115.09) .. controls (338.11,112.94) and (343.35,92.52) .. (328.4,97.24) ;
\draw [shift={(328.4,97.24)}, rotate = 162.47] [color={rgb, 255:red, 0; green, 0; blue, 0 }  ][fill={rgb, 255:red, 0; green, 0; blue, 0 }  ][line width=0.75]      (0, 0) circle [x radius= 2.01, y radius= 2.01]   ;
\draw [shift={(365,115.24)}, rotate = 182.12] [color={rgb, 255:red, 0; green, 0; blue, 0 }  ][line width=0.75]    (6.56,-1.97) .. controls (4.17,-0.84) and (1.99,-0.18) .. (0,0) .. controls (1.99,0.18) and (4.17,0.84) .. (6.56,1.97)   ;
\draw    (229.04,127.41) ;
\draw [shift={(229.04,127.41)}, rotate = 0] [color={rgb, 255:red, 0; green, 0; blue, 0 }  ][fill={rgb, 255:red, 0; green, 0; blue, 0 }  ][line width=0.75]      (0, 0) circle [x radius= 3.35, y radius= 3.35]   ;
\draw  [fill={rgb, 255:red, 206; green, 206; blue, 206 }  ,fill opacity=1 ] (237.12,134.81) .. controls (237.12,126.26) and (244.05,119.33) .. (252.6,119.33) .. controls (261.15,119.33) and (268.08,126.26) .. (268.08,134.81) .. controls (268.08,143.36) and (261.15,150.29) .. (252.6,150.29) .. controls (244.05,150.29) and (237.12,143.36) .. (237.12,134.81) -- cycle ;
\draw    (252.6,134.81) ;
\draw [shift={(252.6,134.81)}, rotate = 0] [color={rgb, 255:red, 0; green, 0; blue, 0 }  ][fill={rgb, 255:red, 0; green, 0; blue, 0 }  ][line width=0.75]      (0, 0) circle [x radius= 3.35, y radius= 3.35]   ;
\draw   (181.69,127.41) .. controls (181.69,101.25) and (202.89,80.06) .. (229.04,80.06) .. controls (255.19,80.06) and (276.39,101.25) .. (276.39,127.41) .. controls (276.39,153.56) and (255.19,174.76) .. (229.04,174.76) .. controls (202.89,174.76) and (181.69,153.56) .. (181.69,127.41) -- cycle ;
\draw    (252.6,134.81) -- (268.08,134.81) ;
\draw [shift={(268.08,134.81)}, rotate = 45] [color={rgb, 255:red, 0; green, 0; blue, 0 }  ][line width=0.75]    (-2.24,0) -- (2.24,0)(0,2.24) -- (0,-2.24)   ;
\draw [shift={(252.6,134.81)}, rotate = 45] [color={rgb, 255:red, 0; green, 0; blue, 0 }  ][line width=0.75]    (-2.24,0) -- (2.24,0)(0,2.24) -- (0,-2.24)   ;
\draw    (237.12,134.81) -- (252.6,134.81) ;
\draw [shift={(252.6,134.81)}, rotate = 45] [color={rgb, 255:red, 0; green, 0; blue, 0 }  ][line width=0.75]    (-2.24,0) -- (2.24,0)(0,2.24) -- (0,-2.24)   ;
\draw [shift={(237.12,134.81)}, rotate = 45] [color={rgb, 255:red, 0; green, 0; blue, 0 }  ][line width=0.75]    (-2.24,0) -- (2.24,0)(0,2.24) -- (0,-2.24)   ;
\draw    (252.6,134.81) -- (252.6,150.29) ;
\draw [shift={(252.6,150.29)}, rotate = 135] [color={rgb, 255:red, 0; green, 0; blue, 0 }  ][line width=0.75]    (-2.24,0) -- (2.24,0)(0,2.24) -- (0,-2.24)   ;
\draw [shift={(252.6,134.81)}, rotate = 135] [color={rgb, 255:red, 0; green, 0; blue, 0 }  ][line width=0.75]    (-2.24,0) -- (2.24,0)(0,2.24) -- (0,-2.24)   ;
\draw    (252.6,134.81) -- (252.6,119.33) ;
\draw [shift={(252.6,119.33)}, rotate = 315] [color={rgb, 255:red, 0; green, 0; blue, 0 }  ][line width=0.75]    (-2.24,0) -- (2.24,0)(0,2.24) -- (0,-2.24)   ;
\draw [shift={(252.6,134.81)}, rotate = 315] [color={rgb, 255:red, 0; green, 0; blue, 0 }  ][line width=0.75]    (-2.24,0) -- (2.24,0)(0,2.24) -- (0,-2.24)   ;
\draw    (182.13,134.81) -- (237.12,134.81) ;
\draw [shift={(237.12,134.81)}, rotate = 45] [color={rgb, 255:red, 0; green, 0; blue, 0 }  ][line width=0.75]    (-2.24,0) -- (2.24,0)(0,2.24) -- (0,-2.24)   ;
\draw [shift={(182.13,134.81)}, rotate = 45] [color={rgb, 255:red, 0; green, 0; blue, 0 }  ][line width=0.75]    (-2.24,0) -- (2.24,0)(0,2.24) -- (0,-2.24)   ;
\draw    (268.08,134.81) -- (275.95,134.81) ;
\draw [shift={(275.95,134.81)}, rotate = 45] [color={rgb, 255:red, 0; green, 0; blue, 0 }  ][line width=0.75]    (-2.24,0) -- (2.24,0)(0,2.24) -- (0,-2.24)   ;
\draw [shift={(268.08,134.81)}, rotate = 45] [color={rgb, 255:red, 0; green, 0; blue, 0 }  ][line width=0.75]    (-2.24,0) -- (2.24,0)(0,2.24) -- (0,-2.24)   ;
\draw    (252.6,119.33) -- (252.6,86.61) ;
\draw [shift={(252.6,86.61)}, rotate = 315] [color={rgb, 255:red, 0; green, 0; blue, 0 }  ][line width=0.75]    (-2.24,0) -- (2.24,0)(0,2.24) -- (0,-2.24)   ;
\draw [shift={(252.6,119.33)}, rotate = 315] [color={rgb, 255:red, 0; green, 0; blue, 0 }  ][line width=0.75]    (-2.24,0) -- (2.24,0)(0,2.24) -- (0,-2.24)   ;
\draw    (252.6,150.29) -- (252.6,168) ;
\draw [shift={(252.6,168)}, rotate = 135] [color={rgb, 255:red, 0; green, 0; blue, 0 }  ][line width=0.75]    (-2.24,0) -- (2.24,0)(0,2.24) -- (0,-2.24)   ;
\draw [shift={(252.6,150.29)}, rotate = 135] [color={rgb, 255:red, 0; green, 0; blue, 0 }  ][line width=0.75]    (-2.24,0) -- (2.24,0)(0,2.24) -- (0,-2.24)   ;
\draw  [color={rgb, 255:red, 155; green, 155; blue, 155 }  ,draw opacity=1 ][dash pattern={on 1.5pt off 2.25pt}] (153.45,150.17) .. controls (140.87,108.42) and (164.53,64.39) .. (206.28,51.81) .. controls (248.03,39.24) and (292.06,62.89) .. (304.64,104.64) .. controls (317.21,146.39) and (293.56,190.43) .. (251.81,203) .. controls (210.06,215.57) and (166.02,191.92) .. (153.45,150.17) -- cycle ;
\draw  [color={rgb, 255:red, 155; green, 155; blue, 155 }  ,draw opacity=1 ][dash pattern={on 1.5pt off 2.25pt}] (142.44,153.49) .. controls (128.03,105.66) and (155.13,55.21) .. (202.96,40.8) .. controls (250.79,26.4) and (301.24,53.49) .. (315.65,101.32) .. controls (330.05,149.15) and (302.95,199.61) .. (255.12,214.01) .. controls (207.29,228.41) and (156.84,201.32) .. (142.44,153.49) -- cycle ;
\draw  [color={rgb, 255:red, 155; green, 155; blue, 155 }  ,draw opacity=1 ][dash pattern={on 1.5pt off 2.25pt}] (164.99,146.7) .. controls (154.33,111.32) and (174.38,74.01) .. (209.75,63.35) .. controls (245.13,52.7) and (282.44,72.74) .. (293.09,108.12) .. controls (303.75,143.49) and (283.71,180.81) .. (248.33,191.46) .. controls (212.96,202.11) and (175.64,182.07) .. (164.99,146.7) -- cycle ;
\draw  [color={rgb, 255:red, 155; green, 155; blue, 155 }  ,draw opacity=1 ][dash pattern={on 1.5pt off 2.25pt}] (175.02,143.67) .. controls (166.04,113.84) and (182.94,82.37) .. (212.77,73.39) .. controls (242.61,64.4) and (274.07,81.31) .. (283.06,111.14) .. controls (292.04,140.97) and (275.14,172.44) .. (245.31,181.42) .. controls (215.48,190.41) and (184.01,173.51) .. (175.02,143.67) -- cycle ;
\draw [line width=0.75]    (182.13,134.4) -- (173.33,134.4) -- (173.66,138.78) -- (175.02,143.67) -- (130.86,156.73) ;
\draw [line width=0.75]    (252.6,86.61) -- (264.12,69.84) -- (254.7,64.8) -- (243.94,61.1) -- (235.2,60.09) -- (222.09,60.09) -- (209.75,63.35) -- (199.89,28.55) ;
\draw [line width=0.75]    (275.95,134.81) -- (307.83,134.81) -- (308.17,127.34) -- (307.5,113.22) -- (304.64,104.3) -- (328.4,97.24) ;
\draw [line width=0.75]    (259.38,226.09) -- (255.12,214.01) -- (268.15,210.07) -- (274.21,205.69) -- (252.6,168) ;

\draw (215.52,115.96) node [anchor=north west][inner sep=0.75pt]  [font=\scriptsize]  {$a$};
\draw (254.6,138.21) node [anchor=north west][inner sep=0.75pt]  [font=\scriptsize]  {$b$};

\end{tikzpicture}

\caption{An example of a construction to modify the centres of the disks in Lemma~\ref{lem:dial} to align with points of a coarse cover, while also letting us `rotate' the edge paths so that they arrive at prescribed angles. }\label{fig:circles}

\end{figure}
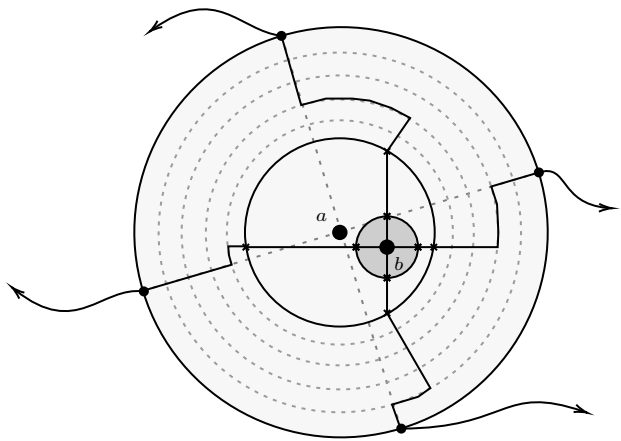

\subsection{Building fat minors using surface subgroups}

We now apply Lemma~\ref{lem:dial} and prove Theorem~\ref{thm:surface-subgroup} of the introduction.

\begin{theorem}\label{thm:surface-subgroup}
	Let $G$ be a one-ended finitely generated group. Suppose $G$ contains an infinite-index surface subgroup $H$. Then $G$ is not asymptotically minor-excluded. 
\end{theorem}

\begin{proof}

Let $X$ be a Cayley graph of $G$. 
Let $F$ be an arbitrary connected finite simplicial graph. We assume without loss of generality that $\Delta(F) \leq 4$, as any graph certainly embeds as a minor into a graph with this property. 
Fix $K > 0$. We will construct a $K$-fat $F$-minor in $X$. Throughout our construction, note that we will make no attempt to optimise our choice of constants.

    Let $Y \subset X$ be a connected tubular neighbourhood of $H$, so that $H \actson Y$ freely and cocompactly. Recall that the inclusion of $Y \into X$ is a coarse embedding. Therefore, we can fix a choice of function $\xi_- : [0, \infty) \to [0,\infty)$ such that 
    $$
        \xi_-(\dist_Y(x,y)) \leq \dist_X(x,y),
    $$
    for all $x,y \in Y$ (we will not need the upper bound function $\xi_+$ here). 
    We assume without loss of generality that $\xi_-$ is increasing, surjective, and every value in $[0,\infty)$ except for 0 has exactly one preimage under $\xi_-$. 
    Since $H$ acts on $Y$ properly discontinuously and cocompactly, we have by the Švarc–Milnor lemma that $Y$ is quasi-isometric to $H$, and thus to either the Euclidean plane $\R^2$ or the hyperbolic plane $\mathbb H^2$ as $H$ is a surface group. Let us assume without loss of generality that we have a quasi-isometry $\varphi : Y \to \R^2$ with quasi-inverse $\psi : \R^2 \to Y$, as the argument is identical in the hyperbolic case.
    Let $\lambda \geq 1 $ now be fixed so that $\varphi$ and $\psi$ are both $\lambda$-quasi-isometries, and that $\psi$ is a $\lambda$-quasi-inverse to $\varphi$ (and vice-versa).
    
    Our first goal is to construct a certain gadget which will allow us to build `coarse crossings'. 
    Using Lemma~\ref{lem:normal-ray-subgroup}, let $\rho : [0, \infty) \to X$ be a one-way infinite geodesic ray based in $Y$ such that $\dist_X(\rho(t), Y) = t$ for every $t \geq 0$. That is, $\rho$ is normal to $Y$.
    Let $C = 10K$. 
	Fix $D > 0$ to be some large constant compared to $C$, $\lambda$. For example, it is sufficient to take 
    $$
    D = 2\lambda^2(10\lambda + \xi_-^{-1}(C))
    $$
		
	Since $X$ is one-ended, it is clear that we can find a path $\alpha :[0,\ell] \to X$ with the following properties:
	\begin{enumerate}
		\item $\alpha(\ell) \in \rho$ and $\alpha(0) \in Y$.
            \item $\alpha$ is disjoint from $N_D(\rho(0))$.
		\item The initial segment $\alpha|_{[0,C]}$ is a geodesic in $X$ which is normal to $Y$.
		\item Otherwise, the terminal segment $\alpha|_{[C+1, \ell]}$ is disjoint from $N_{C}(Y)$.
	\end{enumerate}
	Indeed, this can be achieved by taking a large neighbourhood $N$ of $\rho(0)$ and taking a shortest path from some $\rho(t)$ to $Y$ through $X\setminus N$, then replacing the final segment of this path with a normal. 
	In particular, the union $\alpha \cup \rho$ contains a path $p : [0,s] \to X$ such that
	\begin{enumerate}
		\item $\dist_X(p(0), p(s)) > D$,
		\item The initial and terminal segments $p|_{[0,C]}$ and $p|_{[s-C , s]}$ are disjoint geodesics which are normal to $Y$, and $p$ is otherwise disjoint from the $C$-neighbourhood of $Y$. 
	\end{enumerate}	
    Let $R = \diam_X(p)$.
	Intuitively, the goal now is to use copies of $p$ to build `bridges' which allow us to take fat `planar drawings' of non-planar graphs in $Y$ and turn them into fat minors in $X$. Indeed, from now on we shall refer to each translate of $p$ as a \emph{bridge}.

	For each $h \in H$ write $p_h := h \cdot p$. 
	Let 
	$$
	x_h = \varphi(p_h(0)), \ \ y_h = \varphi(p_h(s)).
	$$
	Note that $\tfrac 1 \lambda D - \lambda < \dist_{\R^2}(x_h, y_h) < \lambda D + \lambda$ for all $h \in H$. 
	Let $L_h$ the line segment in $\R^2$ connecting $x_h$ to $y_h$, and let $m_h \in L_h$ denote its midpoint. Since the action of $H$ on $Y$ is cobounded, it is clear that $N := \{m_h : h \in H\}$ is a coarse cover of $\R^2$.

	Consider some planar, straight-line drawing of $F$, where edges are allowed to cross other edges. We view this drawing itself as a finite simplicial planar graph $P$, where `crossings' in the aforementioned drawing become degree-four vertices of $P$. Note that $\Delta(P) \leq 4$, and $P$ comes naturally equipped with a fixed drawing. 
	If $v \in VP$ corresponds to a crossing in $F$ then we call $v$ a \textit{marked vertex}, and denote the set of marked vertices  by $\widehat VP \subset VP$. 

    Let $r = 3(\lambda D + \lambda)$. Let $K' > 0$ be some fixed constant which is chosen to be sufficiently large compared to $C$, $\lambda$, $r$, and $R$. For example, it is sufficient to take 
    $$
    K' = 4\lambda (10\lambda + \xi_-^{-1}(C + 4R)) + 10r.
    $$
    Using the carefully-drawn lemma (\ref{lem:dial}), let $M \subset \R^2$ be a $K'$-fat $P$-minor where
    \begin{enumerate}

         \item For every vertex $v \in VP$, the branch set $U_v$ is a solid circular disk of radius $r$ based at a point in $N$. 
 
		\item The cyclic order in which the edge paths leave the branch sets agrees with the cyclic  order with which the edges leave the vertices in the given drawing of $P$.

        \item Let $v \in \widehat VP$ be a marked vertex of $P$, and say $U_v$ is centred on $m_h \in N$. Then the edge paths which meet $U_v$ do so at four evenly spaced points. Moreover, if we extend $L_h$ in both directions so that it forms a diameter of $U_v$, then we also assume that two of edge paths in $M$ which meet $U_v$ do so at exactly where this diameter intersects $\partial U_v$.

     \end{enumerate}

     We now need to set up some more notation. If $x_h, y_h, m_h \in U_v$ then let us write $x_v := x_h$, and so on and also write $L_v := L_h$. We also relabel the corresponding bridge $p_v := p_h$. Given $e = uv \in EP$, let $P_e = P_{uv}$ denote the corresponding edge path in $M$. 
     For each marked vertex $v \in \widehat VP$, let $a_v^{1}, \ldots, a_v^4 \in \partial U_v$ be where the edge paths of $M$ meet $U_v$, cyclically ordered so that $a_v^1$, $x_v$, $y_v$, and $a_v^3$ appear on a common diameter in this given order.
     If $P_{uw}$ is an edge path such that that $u$ is a marked vertex, and this path terminates at either $a_u^1$ (or $a_u^3$), then we extend $P_{uw}$ by adjoining the line segment connecting $a_u^1$ to $x_u$ (resp. $a_u^3$ to $y_u$). If $w$ is also a marked vertex then we extend $P_{uw}$ similarly at the other end. Call the resulting extended path $P'_{uw}$. If neither $u$ nor $w$ are marked then just set $P'_{uw} = P_{uw}$. Finally, given marked $v \in \widehat VP$ we let $Q_v$ denote the line segment connecting $a_v^2$ to $a_v^4$. 
     
     We now consider the figure
     $$
     M' = \bigcup_{v \in VP \setminus \widehat VP} U_v \cup \bigcup_{v \in \widehat VP} Q_v \cup \bigcup_{e\in EP} P_e'. 
     $$
     In other words, we delete each $U_v$ where $v \in \widehat VP$ is a marked vertex, then add back line segments connecting $a_v^2$ to $a_v^4$, $a_v^1$ to $x_v$, and $a_v^3$ to $y_v$.
          See Figure~\ref{fig:minor-in-plane} for a cartoon of the construction of $M'$.
     
     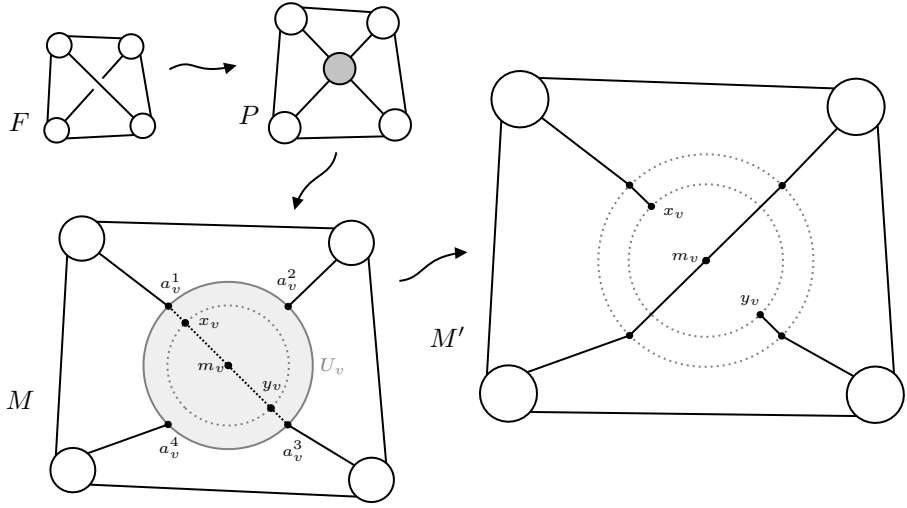
\begin{figure}
     	\centering

\tikzset{every picture/.style={line width=0.75pt}} 

\begin{tikzpicture}[x=0.75pt,y=0.75pt,yscale=-1,xscale=1]

\draw    (119.32,83.37) -- (153.61,81.62) ;
\draw   (110.62,83.87) .. controls (108.08,81.6) and (107.85,77.7) .. (110.12,75.16) .. controls (112.38,72.62) and (116.28,72.39) .. (118.82,74.66) .. controls (121.37,76.93) and (121.59,80.82) .. (119.32,83.37) .. controls (117.06,85.91) and (113.16,86.13) .. (110.62,83.87) -- cycle ;
\draw   (111.92,41.11) .. controls (109.38,38.85) and (109.15,34.95) .. (111.42,32.41) .. controls (113.69,29.87) and (117.58,29.64) .. (120.13,31.91) .. controls (122.67,34.17) and (122.89,38.07) .. (120.63,40.61) .. controls (118.36,43.16) and (114.46,43.38) .. (111.92,41.11) -- cycle ;
\draw   (147.78,41.71) .. controls (145.23,39.44) and (145.01,35.54) .. (147.28,33) .. controls (149.54,30.46) and (153.44,30.24) .. (155.98,32.5) .. controls (158.52,34.77) and (158.75,38.66) .. (156.48,41.21) .. controls (154.22,43.75) and (150.32,43.97) .. (147.78,41.71) -- cycle ;
\draw   (153.61,81.62) .. controls (151.07,79.35) and (150.84,75.45) .. (153.11,72.91) .. controls (155.37,70.37) and (159.27,70.14) .. (161.81,72.41) .. controls (164.36,74.68) and (164.58,78.57) .. (162.31,81.12) .. controls (160.05,83.66) and (156.15,83.88) .. (153.61,81.62) -- cycle ;
\draw    (156.48,41.21) -- (161.81,72.41) ;
\draw    (110.12,75.16) -- (111.92,41.11) ;
\draw    (120.13,31.91) -- (147.28,33) ;
\draw    (120.63,40.61) -- (153.11,72.91) ;
\draw    (118.82,74.66) -- (133.16,58.58) ;
\draw    (137.55,53.19) -- (147.78,41.71) ;
\draw    (234.31,83.35) -- (278.29,82.32) ;
\draw   (223.15,83.69) .. controls (219.97,80.7) and (219.82,75.7) .. (222.81,72.52) .. controls (225.8,69.34) and (230.8,69.19) .. (233.98,72.19) .. controls (237.16,75.18) and (237.31,80.18) .. (234.31,83.35) .. controls (231.32,86.53) and (226.32,86.68) .. (223.15,83.69) -- cycle ;
\draw   (226.32,28.98) .. controls (223.14,25.99) and (222.99,20.99) .. (225.98,17.81) .. controls (228.97,14.64) and (233.97,14.49) .. (237.15,17.48) .. controls (240.32,20.47) and (240.47,25.47) .. (237.48,28.64) .. controls (234.49,31.82) and (229.49,31.97) .. (226.32,28.98) -- cycle ;
\draw   (272.22,31) .. controls (269.04,28.01) and (268.89,23.01) .. (271.88,19.83) .. controls (274.87,16.65) and (279.87,16.5) .. (283.05,19.49) .. controls (286.23,22.49) and (286.38,27.49) .. (283.38,30.66) .. controls (280.39,33.84) and (275.39,33.99) .. (272.22,31) -- cycle ;
\draw   (278.29,82.32) .. controls (275.11,79.32) and (274.96,74.32) .. (277.95,71.15) .. controls (280.94,67.97) and (285.94,67.82) .. (289.12,70.81) .. controls (292.29,73.8) and (292.44,78.8) .. (289.45,81.98) .. controls (286.46,85.16) and (281.46,85.31) .. (278.29,82.32) -- cycle ;
\draw    (283.38,30.66) -- (289.12,70.81) ;
\draw    (222.81,72.52) -- (226.32,28.98) ;
\draw    (237.15,17.48) -- (271.88,19.83) ;
\draw    (237.48,28.64) -- (277.95,71.15) ;
\draw    (233.98,72.19) -- (252.9,52.09) ;
\draw    (258.71,45.34) -- (272.22,31) ;
\draw  [fill={rgb, 255:red, 195; green, 195; blue, 195 }  ,fill opacity=1 ] (250.55,54) .. controls (247.38,51.01) and (247.23,46.01) .. (250.22,42.84) .. controls (253.21,39.66) and (258.21,39.51) .. (261.39,42.5) .. controls (264.56,45.49) and (264.71,50.49) .. (261.72,53.67) .. controls (258.73,56.84) and (253.73,56.99) .. (250.55,54) -- cycle ;
\draw    (171,48.6) .. controls (184.06,41.69) and (181.63,51.73) .. (201.94,47.24) ;
\draw [shift={(204.6,46.6)}, rotate = 165.5] [fill={rgb, 255:red, 0; green, 0; blue, 0 }  ][line width=0.08]  [draw opacity=0] (6.25,-3) -- (0,0) -- (6.25,3) -- cycle    ;
\draw    (254.34,90.34) .. controls (251.55,103.89) and (240.22,100.23) .. (234.29,116.64) ;
\draw [shift={(233.4,119.4)}, rotate = 285.95] [fill={rgb, 255:red, 0; green, 0; blue, 0 }  ][line width=0.08]  [draw opacity=0] (6.25,-3) -- (0,0) -- (6.25,3) -- cycle    ;
\draw  [color={rgb, 255:red, 128; green, 128; blue, 128 }  ,draw opacity=1 ][fill={rgb, 255:red, 239; green, 239; blue, 239 }  ,fill opacity=1 ] (170.34,227.23) .. controls (153.89,210.74) and (153.98,184.1) .. (170.53,167.72) .. controls (187.08,151.34) and (213.84,151.42) .. (230.29,167.91) .. controls (246.74,184.39) and (246.66,211.03) .. (230.1,227.41) .. controls (213.55,243.8) and (186.79,243.71) .. (170.34,227.23) -- cycle ;
\draw   (119.98,142.14) .. controls (115.3,138.11) and (114.79,131.06) .. (118.83,126.4) .. controls (122.88,121.74) and (129.96,121.23) .. (134.64,125.26) .. controls (139.32,129.29) and (139.83,136.34) .. (135.78,141) .. controls (131.73,145.66) and (124.66,146.17) .. (119.98,142.14) -- cycle ;
\draw   (254.54,144.34) .. controls (249.86,140.3) and (249.35,133.26) .. (253.4,128.6) .. controls (257.45,123.94) and (264.52,123.43) .. (269.2,127.46) .. controls (273.88,131.49) and (274.39,138.54) .. (270.35,143.2) .. controls (266.3,147.86) and (259.22,148.37) .. (254.54,144.34) -- cycle ;
\draw   (264.54,263.54) .. controls (259.86,259.5) and (259.35,252.46) .. (263.4,247.8) .. controls (267.45,243.14) and (274.52,242.63) .. (279.2,246.66) .. controls (283.88,250.69) and (284.39,257.74) .. (280.35,262.4) .. controls (276.3,267.06) and (269.22,267.57) .. (264.54,263.54) -- cycle ;
\draw   (115.47,258.64) .. controls (110.79,254.61) and (110.28,247.56) .. (114.32,242.9) .. controls (118.37,238.24) and (125.45,237.73) .. (130.13,241.76) .. controls (134.8,245.79) and (135.32,252.84) .. (131.27,257.5) .. controls (127.22,262.16) and (120.15,262.67) .. (115.47,258.64) -- cycle ;
\draw    (119.98,142.14) -- (114.32,242.9) ;
\draw    (134.64,125.26) -- (253.4,128.6) ;
\draw    (279.2,246.66) -- (270.35,143.2) ;
\draw    (264.54,263.54) -- (131.27,257.5) ;
\draw  [color={rgb, 255:red, 128; green, 128; blue, 128 }  ,draw opacity=1 ][dash pattern={on 0.75pt off 1.5pt}] (178.8,218.85) .. controls (167,207.02) and (167.06,187.9) .. (178.94,176.14) .. controls (190.82,164.39) and (210.02,164.45) .. (221.83,176.28) .. controls (233.64,188.11) and (233.57,207.23) .. (221.69,218.99) .. controls (209.81,230.75) and (190.61,230.69) .. (178.8,218.85) -- cycle ;
\draw  [dash pattern={on 0.75pt off 0.75pt}]  (170.53,167.72) -- (230.1,227.41) ;
\draw [shift={(230.1,227.41)}, rotate = 45.06] [color={rgb, 255:red, 0; green, 0; blue, 0 }  ][fill={rgb, 255:red, 0; green, 0; blue, 0 }  ][line width=0.75]      (0, 0) circle [x radius= 1.34, y radius= 1.34]   ;
\draw [shift={(170.53,167.72)}, rotate = 45.06] [color={rgb, 255:red, 0; green, 0; blue, 0 }  ][fill={rgb, 255:red, 0; green, 0; blue, 0 }  ][line width=0.75]      (0, 0) circle [x radius= 1.34, y radius= 1.34]   ;
\draw    (130.13,241.76) -- (170.34,227.23) ;
\draw [shift={(170.34,227.23)}, rotate = 340.13] [color={rgb, 255:red, 0; green, 0; blue, 0 }  ][fill={rgb, 255:red, 0; green, 0; blue, 0 }  ][line width=0.75]      (0, 0) circle [x radius= 1.34, y radius= 1.34]   ;
\draw    (135.78,141) -- (170.34,167.72) ;
\draw    (230.1,227.41) -- (263.4,247.8) ;
\draw    (221.69,218.99) -- (221.69,218.99) ;
\draw [shift={(221.69,218.99)}, rotate = 270.26] [color={rgb, 255:red, 0; green, 0; blue, 0 }  ][fill={rgb, 255:red, 0; green, 0; blue, 0 }  ][line width=0.75]      (0, 0) circle [x radius= 1.34, y radius= 1.34]   ;
\draw [shift={(221.69,218.99)}, rotate = 270.26] [color={rgb, 255:red, 0; green, 0; blue, 0 }  ][fill={rgb, 255:red, 0; green, 0; blue, 0 }  ][line width=0.75]      (0, 0) circle [x radius= 1.34, y radius= 1.34]   ;
\draw    (178.94,176.14) -- (178.94,176.14) ;
\draw [shift={(178.94,176.14)}, rotate = 270.26] [color={rgb, 255:red, 0; green, 0; blue, 0 }  ][fill={rgb, 255:red, 0; green, 0; blue, 0 }  ][line width=0.75]      (0, 0) circle [x radius= 1.34, y radius= 1.34]   ;
\draw [shift={(178.94,176.14)}, rotate = 270.26] [color={rgb, 255:red, 0; green, 0; blue, 0 }  ][fill={rgb, 255:red, 0; green, 0; blue, 0 }  ][line width=0.75]      (0, 0) circle [x radius= 1.34, y radius= 1.34]   ;
\draw    (254.54,144.34) -- (230.29,167.91) ;
\draw [shift={(230.29,167.91)}, rotate = 135.82] [color={rgb, 255:red, 0; green, 0; blue, 0 }  ][fill={rgb, 255:red, 0; green, 0; blue, 0 }  ][line width=0.75]      (0, 0) circle [x radius= 1.34, y radius= 1.34]   ;
\draw  [dash pattern={on 0.75pt off 0.75pt}]  (200.32,197.57) -- (200.5,197.57) ;
\draw [shift={(200.5,197.57)}, rotate = 360] [color={rgb, 255:red, 0; green, 0; blue, 0 }  ][fill={rgb, 255:red, 0; green, 0; blue, 0 }  ][line width=0.75]      (0, 0) circle [x radius= 1.34, y radius= 1.34]   ;
\draw [shift={(200.32,197.57)}, rotate = 360] [color={rgb, 255:red, 0; green, 0; blue, 0 }  ][fill={rgb, 255:red, 0; green, 0; blue, 0 }  ][line width=0.75]      (0, 0) circle [x radius= 1.34, y radius= 1.34]   ;
\draw  [color={rgb, 255:red, 128; green, 128; blue, 128 }  ,draw opacity=1 ][dash pattern={on 0.75pt off 1.5pt}] (400.68,182.48) .. controls (379.78,161.53) and (379.88,127.67) .. (400.92,106.86) .. controls (421.95,86.04) and (455.95,86.15) .. (476.86,107.09) .. controls (497.76,128.04) and (497.65,161.9) .. (476.62,182.71) .. controls (455.58,203.53) and (421.58,203.42) .. (400.68,182.48) -- cycle ;
\draw   (336.68,74.35) .. controls (330.73,69.23) and (330.08,60.28) .. (335.23,54.35) .. controls (340.37,48.43) and (349.36,47.79) .. (355.31,52.91) .. controls (361.25,58.03) and (361.9,66.98) .. (356.76,72.9) .. controls (351.62,78.82) and (342.63,79.47) .. (336.68,74.35) -- cycle ;
\draw   (504.45,78.2) .. controls (498.51,73.07) and (497.86,64.12) .. (503,58.2) .. controls (508.14,52.28) and (517.13,51.63) .. (523.08,56.75) .. controls (529.03,61.87) and (529.68,70.83) .. (524.53,76.75) .. controls (519.39,82.67) and (510.4,83.32) .. (504.45,78.2) -- cycle ;
\draw   (514.05,223) .. controls (508.11,217.87) and (507.46,208.92) .. (512.6,203) .. controls (517.74,197.08) and (526.73,196.43) .. (532.68,201.55) .. controls (538.63,206.67) and (539.28,215.63) .. (534.13,221.55) .. controls (528.99,227.47) and (520,228.12) .. (514.05,223) -- cycle ;
\draw   (330.95,222.39) .. controls (325,217.27) and (324.35,208.32) .. (329.49,202.4) .. controls (334.64,196.47) and (343.63,195.83) .. (349.57,200.95) .. controls (355.52,206.07) and (356.17,215.02) .. (351.03,220.94) .. controls (345.88,226.87) and (336.89,227.51) .. (330.95,222.39) -- cycle ;
\draw    (336.68,74.35) -- (329.49,202.4) ;
\draw    (355.31,52.91) -- (503,58.2) ;
\draw    (532.68,201.55) -- (524.53,76.75) ;
\draw    (514.05,223) -- (351.03,220.94) ;
\draw  [color={rgb, 255:red, 128; green, 128; blue, 128 }  ,draw opacity=1 ][dash pattern={on 0.75pt off 1.5pt}] (411.43,171.84) .. controls (396.43,156.8) and (396.5,132.5) .. (411.6,117.56) .. controls (426.7,102.62) and (451.1,102.7) .. (466.1,117.73) .. controls (481.11,132.77) and (481.03,157.07) .. (465.93,172.01) .. controls (450.84,186.95) and (426.43,186.87) .. (411.43,171.84) -- cycle ;
\draw    (349.57,200.95) -- (400.68,182.48) ;
\draw    (356.76,72.9) -- (400.68,106.86) ;
\draw    (476.62,182.71) -- (512.6,203) ;
\draw    (476.62,182.71) -- (465.93,172.01) ;
\draw [shift={(465.93,172.01)}, rotate = 225.06] [color={rgb, 255:red, 0; green, 0; blue, 0 }  ][fill={rgb, 255:red, 0; green, 0; blue, 0 }  ][line width=0.75]      (0, 0) circle [x radius= 1.34, y radius= 1.34]   ;
\draw [shift={(476.62,182.71)}, rotate = 225.06] [color={rgb, 255:red, 0; green, 0; blue, 0 }  ][fill={rgb, 255:red, 0; green, 0; blue, 0 }  ][line width=0.75]      (0, 0) circle [x radius= 1.34, y radius= 1.34]   ;
\draw    (411.6,117.56) -- (400.68,106.86) ;
\draw [shift={(400.68,106.86)}, rotate = 224.43] [color={rgb, 255:red, 0; green, 0; blue, 0 }  ][fill={rgb, 255:red, 0; green, 0; blue, 0 }  ][line width=0.75]      (0, 0) circle [x radius= 1.34, y radius= 1.34]   ;
\draw [shift={(411.6,117.56)}, rotate = 224.43] [color={rgb, 255:red, 0; green, 0; blue, 0 }  ][fill={rgb, 255:red, 0; green, 0; blue, 0 }  ][line width=0.75]      (0, 0) circle [x radius= 1.34, y radius= 1.34]   ;
\draw    (504.45,78.2) -- (476.86,107.1) ;
\draw [color={rgb, 255:red, 0; green, 0; blue, 0 }  ,draw opacity=1 ]   (476.86,107.1) -- (400.68,182.48) ;
\draw [shift={(400.68,182.48)}, rotate = 135.3] [color={rgb, 255:red, 0; green, 0; blue, 0 }  ,draw opacity=1 ][fill={rgb, 255:red, 0; green, 0; blue, 0 }  ,fill opacity=1 ][line width=0.75]      (0, 0) circle [x radius= 1.34, y radius= 1.34]   ;
\draw [shift={(476.86,107.1)}, rotate = 135.3] [color={rgb, 255:red, 0; green, 0; blue, 0 }  ,draw opacity=1 ][fill={rgb, 255:red, 0; green, 0; blue, 0 }  ,fill opacity=1 ][line width=0.75]      (0, 0) circle [x radius= 1.34, y radius= 1.34]   ;
\draw  [dash pattern={on 0.75pt off 0.75pt}]  (438.77,144.79) -- (439.01,144.79) ;
\draw [shift={(439.01,144.79)}, rotate = 360] [color={rgb, 255:red, 0; green, 0; blue, 0 }  ][fill={rgb, 255:red, 0; green, 0; blue, 0 }  ][line width=0.75]      (0, 0) circle [x radius= 1.34, y radius= 1.34]   ;
\draw [shift={(438.77,144.79)}, rotate = 360] [color={rgb, 255:red, 0; green, 0; blue, 0 }  ][fill={rgb, 255:red, 0; green, 0; blue, 0 }  ][line width=0.75]      (0, 0) circle [x radius= 1.34, y radius= 1.34]   ;
\draw    (285.8,155) .. controls (298.41,153.85) and (296.78,143.22) .. (316.82,140.87) ;
\draw [shift={(319.8,140.6)}, rotate = 176.05] [fill={rgb, 255:red, 0; green, 0; blue, 0 }  ][line width=0.08]  [draw opacity=0] (6.25,-3) -- (0,0) -- (6.25,3) -- cycle    ;

\draw (204,65.4) node [anchor=north west][inner sep=0.75pt]    {$P$};
\draw (89.36,69) node [anchor=north west][inner sep=0.75pt]    {$F$};
\draw (183.4,194.8) node [anchor=north west][inner sep=0.75pt]  [font=\tiny]  {$m_{v}$};
\draw (184.07,172.42) node [anchor=north west][inner sep=0.75pt]  [font=\tiny]  {$x_{v}$};
\draw (215.56,203.76) node [anchor=north west][inner sep=0.75pt]  [font=\tiny]  {$y_{v}$};
\draw (244.65,192.42) node [anchor=north west][inner sep=0.75pt]  [font=\scriptsize,color={rgb, 255:red, 128; green, 128; blue, 128 }  ,opacity=1 ]  {$U_{v}$};
\draw (87.84,208.96) node [anchor=north west][inner sep=0.75pt]    {$M$};
\draw (164.9,149.53) node [anchor=north west][inner sep=0.75pt]  [font=\tiny]  {$a_{v}^{1}$};
\draw (222.91,147.85) node [anchor=north west][inner sep=0.75pt]  [font=\tiny]  {$a_{v}^{2}$};
\draw (225.9,233.84) node [anchor=north west][inner sep=0.75pt]  [font=\tiny]  {$a_{v}^{3}$};
\draw (164.24,232.36) node [anchor=north west][inner sep=0.75pt]  [font=\tiny]  {$a_{v}^{4}$};
\draw (420.1,139.32) node [anchor=north west][inner sep=0.75pt]  [font=\tiny]  {$m_{v}$};
\draw (415.94,117.06) node [anchor=north west][inner sep=0.75pt]  [font=\tiny]  {$x_{v}$};
\draw (454.16,160.7) node [anchor=north west][inner sep=0.75pt]  [font=\tiny]  {$y_{v}$};
\draw (299.25,176.5) node [anchor=north west][inner sep=0.75pt]    {$M'$};

\end{tikzpicture}

     	\caption{The construction of $M' \subset \R^2$. In this example, $F = K_4$ and $P$ is constructed from the shown non-planar drawing of $F$.}
     	\label{fig:minor-in-plane} 
     \end{figure} 
     
     Our goal now is to push $M'$ through the quasi-isometry $\psi : \R^2 \to Y$ and use this to construct a fat copy of $F$ in $X$. 
     
     For each $e \in EP$ let $C_e \subset P'_e$ be a 1-coarse cover of $P'_e$ which contains the endpoints of $P_e'$. For a fixed $e$, enumerate $C_e = \{b_1, \ldots, b_l\}$ in the order that they appear. By joining $\psi(b_i)$ with a geodesic in $Y$ to $\psi(b_{i+1})$ for each $i$, we form a path $A_e$ in $Y$ connecting the $\psi$-images of $P_e'$, which is contained in the $2\lambda$-neighbourhood of $\psi(C_e)$, with respect to the intrinsic metric $\dist_Y$ of $Y$. We similarly repeat this construction with each $Q_v$ where $v \in \widehat VP$, and obtain a similar path $B_v$. 
     
     Given $v \in \widehat VP$, let us write $a_v = p_v(0)$, $b_v = p_v(s)$, $a'_v = \psi(x_v)$ $b'_v = \psi(y_v)$. Note that 
     $$
     \dist_Y(a_v, a'_v) \leq 2\lambda, \ \ \dist_Y(b_v, b'_v) \leq 2\lambda.
     $$
     Thus, we extend any path $A_{uv}$ $u \in VP$ terminating at $a_v'$ (resp. $b_v'$) with a geodesic of length at most $2\lambda$ so that it terminates at $a_v$ (resp. $b_v$). 
     If $v \in VP\setminus \widehat VP$ is not marked, then let $U_v'$ denote the $2\lambda$-neighbourhood of $\psi(U_v)$ in $Y$, so $U_v'$ is connected and contains the endpoints of all the paths $A_{e}$ where $e$ abuts $v$ in $P$. 
     
     We now consider the subset of $X$ given by 
     $$
     M'' = \bigcup_{v \in VP \setminus \widehat VP} U_v' \cup \bigcup_{e\in EP} A_e \cup \bigcup_{v \in \widehat VP} B_v  \cup \bigcup_{v \in \widehat VP} p_v.  
     $$
     
     By comparing $M''$ with the non-planar drawing of $F$ we started with, it is now straightforward to verify that $M''$ decomposes as a $K$-fat $F$-minor. We need only verify a few inequalities. For example, consider $v,u \in VF$. These correspond to unmarked vertices in $VP \setminus \widehat VP$, and the corresponding branch sets are thus $U_v', U_u' \subset M''$.
     We may lower bound the infimal distance $\dist_X(U_v', U_u')$ via:
     \begin{align*}
         \dist_X(U_v', U_u') &\geq \xi_-(\dist_Y(U_v', U_u')) \\
         &\geq  \xi_-(  \dist_Y(\psi(U_v), \psi(U_u))  -4\lambda)  \\
         &\geq  \xi_-(  \tfrac 1 \lambda \dist_{\R^2}(U_v, U_u) - \lambda -4\lambda)\\ 
         &\geq  \xi_-(  \tfrac 1 \lambda K' - \lambda -4\lambda) \\
         &=  \xi_-(  \tfrac 1 \lambda (10r + 4\lambda (10\lambda + \xi_-^{-1}(C+4R))) - \lambda -4\lambda) \\
         &>  \xi_-(  \xi_-^{-1}(C+4R)) \\
         &=  C+4R \\
         &> K.
     \end{align*}
    Thus, our branch sets are all sufficiently spaced-out. 
    We should also check that our bridges stay far away from everything they are supposed to. For example, let $v \in \widehat VP$ be a marked vertex. One must verify that $B_v$ and $p_v$ are sufficiently far apart. 

    Recall that the bridge $p_v$ has the form of two geodesics which are normal to $Y$, based at $a_v$ and $b_v$, connected at the top by a path which lies at distance at least $C$ from $Y$ everywhere. 
    We thus may assume without loss of generality that $\dist_X(b_v, B_v)\geq \dist_X(a_v, B_v)$, as if a vertex in $Y$ is close to the bridge $B_v$ then its closest point is either $a_v$, $b_v$, or a point on the `top' of the bridge. We then bound $\dist_X(B_v, p_v)$ as follows:
    \begin{align*}
        \dist_X(a_v, B_v) &\geq \xi_-(\dist_Y(a_v, B_v)) \\
        & \geq \xi_-(\dist_Y(\psi(x_h), \psi(Q_v)) - 4\lambda) \\
        & \geq \xi_-(\tfrac 1 \lambda \dist_{\R^2}(x_h), Q_v) - \lambda - 4\lambda) \\
        & \geq \xi_-(\tfrac 1 \lambda (\tfrac 1 \lambda \tfrac D 2 - \lambda) - \lambda - 4\lambda) \\
        & = \xi_-( \tfrac {D} {2\lambda^2} - 1 - 5\lambda) \\
        & > \xi_-(\xi_-^{-1}(C)) \\
        &> K.
    \end{align*}
    This implies that $\dist_X(p_v, B_v) > K$. 
    
    There are, of course, several other inequalities to check here, all of which follow similarly to the above. For example, one can check that the bridges $p_v$ are pairwise far apart from each other, due to the dependence of $K'$ on $R$. One should also check that the $A_e$ are pairwise far apart from each other, and so on. We leave these as an exercise to the keen reader.
    
    See Figure~\ref{fig:surface-final} for a cartoon of this completed construction. 
\end{proof}

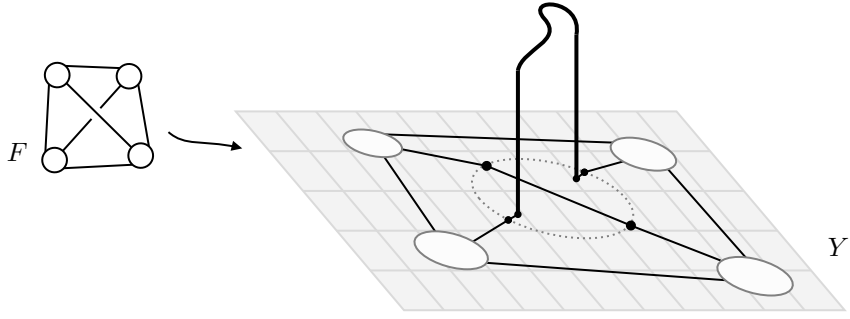
\begin{figure}
	\centering

\tikzset{every picture/.style={line width=0.75pt}} 

\begin{tikzpicture}[x=0.75pt,y=0.75pt,yscale=-1,xscale=1]

\draw  [draw opacity=0][fill={rgb, 255:red, 243; green, 243; blue, 243 }  ,fill opacity=1 ] (188.56,103.2) -- (408.56,103.2) -- (492.53,203.2) -- (272.53,203.2) -- cycle ; \draw  [color={rgb, 255:red, 218; green, 218; blue, 218 }  ,draw opacity=1 ] (208.56,103.2) -- (292.53,203.2)(228.56,103.2) -- (312.53,203.2)(248.56,103.2) -- (332.53,203.2)(268.56,103.2) -- (352.53,203.2)(288.56,103.2) -- (372.53,203.2)(308.56,103.2) -- (392.53,203.2)(328.56,103.2) -- (412.53,203.2)(348.56,103.2) -- (432.53,203.2)(368.56,103.2) -- (452.53,203.2)(388.56,103.2) -- (472.53,203.2) ; \draw  [color={rgb, 255:red, 218; green, 218; blue, 218 }  ,draw opacity=1 ] (205.35,123.2) -- (425.35,123.2)(222.15,143.2) -- (442.15,143.2)(238.94,163.2) -- (458.94,163.2)(255.73,183.2) -- (475.73,183.2) ; \draw  [color={rgb, 255:red, 218; green, 218; blue, 218 }  ,draw opacity=1 ] (188.56,103.2) -- (408.56,103.2) -- (492.53,203.2) -- (272.53,203.2) -- cycle ;
\draw    (102.72,131.77) -- (137.01,130.02) ;
\draw   (94.02,132.27) .. controls (91.48,130) and (91.25,126.1) .. (93.52,123.56) .. controls (95.78,121.02) and (99.68,120.79) .. (102.22,123.06) .. controls (104.77,125.33) and (104.99,129.22) .. (102.72,131.77) .. controls (100.46,134.31) and (96.56,134.53) .. (94.02,132.27) -- cycle ;
\draw   (95.32,89.51) .. controls (92.78,87.25) and (92.55,83.35) .. (94.82,80.81) .. controls (97.09,78.27) and (100.98,78.04) .. (103.53,80.31) .. controls (106.07,82.57) and (106.29,86.47) .. (104.03,89.01) .. controls (101.76,91.56) and (97.86,91.78) .. (95.32,89.51) -- cycle ;
\draw   (131.18,90.11) .. controls (128.63,87.84) and (128.41,83.94) .. (130.68,81.4) .. controls (132.94,78.86) and (136.84,78.64) .. (139.38,80.9) .. controls (141.92,83.17) and (142.15,87.06) .. (139.88,89.61) .. controls (137.62,92.15) and (133.72,92.37) .. (131.18,90.11) -- cycle ;
\draw   (137.01,130.02) .. controls (134.47,127.75) and (134.24,123.85) .. (136.51,121.31) .. controls (138.77,118.77) and (142.67,118.54) .. (145.21,120.81) .. controls (147.76,123.08) and (147.98,126.97) .. (145.71,129.52) .. controls (143.45,132.06) and (139.55,132.28) .. (137.01,130.02) -- cycle ;
\draw    (139.88,89.61) -- (145.21,120.81) ;
\draw    (93.52,123.56) -- (95.32,89.51) ;
\draw    (103.53,80.31) -- (130.68,81.4) ;
\draw    (104.03,89.01) -- (136.51,121.31) ;
\draw    (102.22,123.06) -- (116.56,106.98) ;
\draw    (120.95,101.59) -- (131.18,90.11) ;
\draw    (154.96,113.6) .. controls (163.36,121.01) and (171.72,117.41) .. (189,121.15) ;
\draw [shift={(191.8,121.8)}, rotate = 194.04] [fill={rgb, 255:red, 0; green, 0; blue, 0 }  ][line width=0.08]  [draw opacity=0] (5.36,-2.57) -- (0,0) -- (5.36,2.57) -- cycle    ;
\draw  [color={rgb, 255:red, 128; green, 128; blue, 128 }  ,draw opacity=1 ][dash pattern={on 0.75pt off 1.5pt}] (310.12,147.05) .. controls (300.88,136.04) and (309.77,127.11) .. (329.99,127.11) .. controls (350.2,127.11) and (374.08,136.04) .. (383.33,147.05) .. controls (392.57,158.07) and (383.68,167) .. (363.46,167) .. controls (343.24,167) and (319.36,158.07) .. (310.12,147.05) -- cycle ;
\draw    (262.73,126.2) -- (288.08,163.5) ;
\draw    (303.4,178.6) -- (439.8,188.6) ;
\draw    (401.8,129.4) -- (446.2,180.2) ;
\draw    (258.6,114.6) -- (380.2,119) ;
\draw    (271.31,122.2) -- (313.8,130.6) ;
\draw    (313.8,130.6) -- (385.8,160.6) ;
\draw [shift={(313.8,130.6)}, rotate = 22.62] [color={rgb, 255:red, 0; green, 0; blue, 0 }  ][fill={rgb, 255:red, 0; green, 0; blue, 0 }  ][line width=0.75]      (0, 0) circle [x radius= 2.01, y radius= 2.01]   ;
\draw    (385.8,160.6) -- (430.6,178.2) ;
\draw [shift={(385.8,160.6)}, rotate = 21.45] [color={rgb, 255:red, 0; green, 0; blue, 0 }  ][fill={rgb, 255:red, 0; green, 0; blue, 0 }  ][line width=0.75]      (0, 0) circle [x radius= 2.01, y radius= 2.01]   ;
\draw  [color={rgb, 255:red, 128; green, 128; blue, 128 }  ,draw opacity=1 ][fill={rgb, 255:red, 252; green, 252; blue, 252 }  ,fill opacity=1 ] (430.8,186.1) .. controls (426.3,180.74) and (430.26,176.4) .. (439.64,176.4) .. controls (449.01,176.4) and (460.25,180.74) .. (464.75,186.1) .. controls (469.25,191.46) and (465.29,195.8) .. (455.91,195.8) .. controls (446.54,195.8) and (435.3,191.46) .. (430.8,186.1) -- cycle ;
\draw  [color={rgb, 255:red, 128; green, 128; blue, 128 }  ,draw opacity=1 ][fill={rgb, 255:red, 252; green, 252; blue, 252 }  ,fill opacity=1 ] (243.43,119.25) .. controls (240.21,115.41) and (243.63,112.3) .. (251.07,112.3) .. controls (258.51,112.3) and (267.15,115.41) .. (270.37,119.25) .. controls (273.59,123.09) and (270.17,126.2) .. (262.73,126.2) .. controls (255.29,126.2) and (246.65,123.09) .. (243.43,119.25) -- cycle ;
\draw    (324.6,157.8) -- (303.8,170.6) ;
\draw [shift={(324.6,157.8)}, rotate = 148.39] [color={rgb, 255:red, 0; green, 0; blue, 0 }  ][fill={rgb, 255:red, 0; green, 0; blue, 0 }  ][line width=0.75]      (0, 0) circle [x radius= 1.34, y radius= 1.34]   ;
\draw    (362.6,133.8) -- (388.2,126.6) ;
\draw [shift={(362.6,133.8)}, rotate = 344.29] [color={rgb, 255:red, 0; green, 0; blue, 0 }  ][fill={rgb, 255:red, 0; green, 0; blue, 0 }  ][line width=0.75]      (0, 0) circle [x radius= 1.34, y radius= 1.34]   ;
\draw    (329.4,155) -- (324.6,157.8) ;
\draw [shift={(324.6,157.8)}, rotate = 149.74] [color={rgb, 255:red, 0; green, 0; blue, 0 }  ][fill={rgb, 255:red, 0; green, 0; blue, 0 }  ][line width=0.75]      (0, 0) circle [x radius= 1.34, y radius= 1.34]   ;
\draw [shift={(329.4,155)}, rotate = 149.74] [color={rgb, 255:red, 0; green, 0; blue, 0 }  ][fill={rgb, 255:red, 0; green, 0; blue, 0 }  ][line width=0.75]      (0, 0) circle [x radius= 1.34, y radius= 1.34]   ;
\draw    (362.6,133.8) -- (358.6,137) ;
\draw [shift={(358.6,137)}, rotate = 141.34] [color={rgb, 255:red, 0; green, 0; blue, 0 }  ][fill={rgb, 255:red, 0; green, 0; blue, 0 }  ][line width=0.75]      (0, 0) circle [x radius= 1.34, y radius= 1.34]   ;
\draw [shift={(362.6,133.8)}, rotate = 141.34] [color={rgb, 255:red, 0; green, 0; blue, 0 }  ][fill={rgb, 255:red, 0; green, 0; blue, 0 }  ][line width=0.75]      (0, 0) circle [x radius= 1.34, y radius= 1.34]   ;
\draw [line width=1.5]    (329.4,82.6) -- (329.4,155) ;
\draw [line width=1.5]    (358.6,64.6) -- (358.6,137) ;
\draw [line width=1.5]    (329.4,82.6) .. controls (331.4,71.8) and (352.6,67.8) .. (342.6,56.6) .. controls (332.6,45.4) and (361.8,46.6) .. (358.6,64.6) ;
\draw  [color={rgb, 255:red, 128; green, 128; blue, 128 }  ,draw opacity=1 ][fill={rgb, 255:red, 252; green, 252; blue, 252 }  ,fill opacity=1 ] (377.23,124.7) .. controls (373.38,120.12) and (376.89,116.4) .. (385.07,116.4) .. controls (393.24,116.4) and (402.99,120.12) .. (406.84,124.7) .. controls (410.68,129.28) and (407.17,133) .. (398.99,133) .. controls (390.82,133) and (381.07,129.28) .. (377.23,124.7) -- cycle ;
\draw  [color={rgb, 255:red, 128; green, 128; blue, 128 }  ,draw opacity=1 ][fill={rgb, 255:red, 252; green, 252; blue, 252 }  ,fill opacity=1 ] (279.61,173.05) .. controls (275.18,167.78) and (278.98,163.5) .. (288.08,163.5) .. controls (297.19,163.5) and (308.16,167.78) .. (312.59,173.05) .. controls (317.01,178.32) and (313.22,182.6) .. (304.11,182.6) .. controls (295,182.6) and (284.03,178.32) .. (279.61,173.05) -- cycle ;

\draw (482.57,165.2) node [anchor=north west][inner sep=0.75pt]    {$Y$};
\draw (72.76,117.4) node [anchor=north west][inner sep=0.75pt]    {$F$};

\end{tikzpicture}

	\caption{Cartoon of the fat $F$-minor $M''$. Here $F = K_4$ and the construction proceeds using the shown non-planar drawing.}\label{fig:surface-final}
\end{figure}

\section{Finitely presented groups}\label{sec:fp}

In this section we construct fat minors in finitely presented groups. Throughout this section, $G$ will denote a one-ended finitely presented group which is not a surface group. Let $X$ denote some choice of Cayley graph of $G$. 

\subsection{Idea of proof} 

Given a one-ended finitely presented group $G$, we will roughly proceed as follows. By applying known structure results for finitely presented groups together with Theorems~\ref{thm:descending-chain} and \ref{thm:surface-subgroup}, we can restrict to the case where $G$ does not split over a two-ended subgroup. Under this assumption, we proceed roughly as follows:

\begin{enumerate}
	\item By a theorem of Papasoglu \cite{papasoglu2005quasi}, we know that no neighbourhood of a bi-infinite geodesic can separate our Cayley graph into deep pieces.
	\item By the Arzela--Ascoli theorem, we see that for every $K > 0$ there exists some $R > 0$ such that the $K$-neighbourhood of any long geodesic segment cannot `nicely' separate the $R$-ball about its midpoint. 
	\item We use this to build a fat $K_m$, where the vertices are `long geodesics'. The previous observation allows us to redirect the edge paths so that they avoid any branch sets they aren't meant to meet.
\end{enumerate}
See Figure~\ref{fig:main-construction} below for a cartoon of the final construction. 

We need to be a bit careful about what we mean for a geodesic to `nicely' separate a ball. This is the topic of the next subsection. 

\subsection{Local coarse separation}

We begin with the following definition, which is useful for shorthand. 

\begin{definition}[Coarsely perpendicular]
    Let $A, B \subset X$ be  subgraphs, $x_0 \in A$, and $L \geq 1$. We say that $B$ is \textit{$L$-coarsely perpendicular to $A$ at $x_0$} if 
    $$
    \dist(b,x_0) \leq L \dist(b, A) + L,
    $$
    for all $b \in B$. 
\end{definition}

Intuitively, $B$ is coarsely perpendicular to $A$ at $x_0$ if $x_0$ is approximately the nearest point in $A$ to every $b \in B$, up to some linear error.
We now have the following key definition. 

\begin{definition}[Local coarse separation]
    Let $R >  r > K > 0$ and $L \geq 0$. Let $\rho \subset X$ be a geodesic, and let $x_0 \in \rho$ be a vertex which lies at distance at least $5R$ from the endpoints of $\rho$. 
    
    Then $\rho$ is said to be \textit{$(R, r, L, K)$-locally coarsely separating at $x_0$} if the complement
    $$
    N_{R}(x_0) \setminus N_K(\rho)
    $$
    contains two distinct connected components $U_1$, $U_2$ such that for $i =1,2$ we have:
    \begin{enumerate}
        \item the subgraph $U_i$ contains a path $p_i$ connecting $S_r(x_0)$ to $S_R(x_0)$, and 

        \item  the path $p_i$ is $L$-coarsely perpendicular to $\rho$.  
    \end{enumerate}
\end{definition}

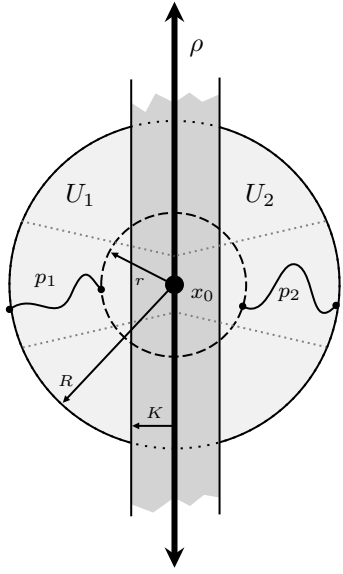
\begin{figure}[t]
    \centering

\tikzset{every picture/.style={line width=0.75pt}} 

\begin{tikzpicture}[x=0.75pt,y=0.75pt,yscale=-1,xscale=1]

\draw  [fill={rgb, 255:red, 241; green, 241; blue, 241 }  ,fill opacity=1 ] (239.57,148.4) .. controls (239.57,102.9) and (276.46,66.01) .. (321.96,66.01) .. controls (367.46,66.01) and (404.35,102.9) .. (404.35,148.4) .. controls (404.35,193.9) and (367.46,230.79) .. (321.96,230.79) .. controls (276.46,230.79) and (239.57,193.9) .. (239.57,148.4) -- cycle ;
\draw  [draw opacity=0][fill={rgb, 255:red, 211; green, 211; blue, 211 }  ,fill opacity=1 ] (329.41,51.7) -- (335.84,56.4) -- (344.31,52.64) -- (344.6,255.49) -- (336.71,254.08) -- (329.7,259.71) -- (321.82,255.49) -- (311.3,263) -- (300.2,258.77) -- (300.35,56.87) -- (307.5,61.56) -- (312.47,52.17) -- (321.82,56.87) -- cycle ;
\draw [line width=2.25]    (321.96,11) -- (321.96,285.8) ;
\draw [shift={(321.96,290.8)}, rotate = 270] [fill={rgb, 255:red, 0; green, 0; blue, 0 }  ][line width=0.08]  [draw opacity=0] (10.36,-4.98) -- (0,0) -- (10.36,4.98) -- (6.88,0) -- cycle    ;
\draw [shift={(321.96,148.4)}, rotate = 90] [color={rgb, 255:red, 0; green, 0; blue, 0 }  ][fill={rgb, 255:red, 0; green, 0; blue, 0 }  ][line width=2.25]      (0, 0) circle [x radius= 3.22, y radius= 3.22]   ;
\draw [shift={(321.96,6)}, rotate = 90] [fill={rgb, 255:red, 0; green, 0; blue, 0 }  ][line width=0.08]  [draw opacity=0] (10.36,-4.98) -- (0,0) -- (10.36,4.98) -- (6.88,0) -- cycle    ;
\draw    (300.35,45.13) -- (300.35,264.88) ;
\draw    (344.45,44.19) -- (344.45,263.94) ;
\draw  [dash pattern={on 3.75pt off 1.5pt}] (285.39,148.4) .. controls (285.39,128.42) and (301.59,112.23) .. (321.56,112.23) .. controls (341.54,112.23) and (357.73,128.42) .. (357.73,148.4) .. controls (357.73,168.38) and (341.54,184.57) .. (321.56,184.57) .. controls (301.59,184.57) and (285.39,168.38) .. (285.39,148.4) -- cycle ;
\draw [color={rgb, 255:red, 128; green, 128; blue, 128 }  ,draw opacity=1 ] [dash pattern={on 0.75pt off 1.5pt}]  (245.48,116.4) -- (321.91,133.97) ;
\draw [color={rgb, 255:red, 128; green, 128; blue, 128 }  ,draw opacity=1 ] [dash pattern={on 0.75pt off 1.5pt}]  (246.2,180) -- (320.83,162.63) ;
\draw [color={rgb, 255:red, 128; green, 128; blue, 128 }  ,draw opacity=1 ] [dash pattern={on 0.75pt off 1.5pt}]  (397.33,116.4) -- (321.91,133.97) ;
\draw [color={rgb, 255:red, 128; green, 128; blue, 128 }  ,draw opacity=1 ] [dash pattern={on 0.75pt off 1.5pt}]  (398.2,179.87) -- (322.77,162.3) ;
\draw    (239.6,160.8) .. controls (252.6,151.2) and (263.4,169.6) .. (270.6,152.8) .. controls (277.8,136) and (280.2,145.6) .. (285.4,150.8) ;
\draw [shift={(285.4,150.8)}, rotate = 45] [color={rgb, 255:red, 0; green, 0; blue, 0 }  ][fill={rgb, 255:red, 0; green, 0; blue, 0 }  ][line width=0.75]      (0, 0) circle [x radius= 1.34, y radius= 1.34]   ;
\draw [shift={(239.6,160.8)}, rotate = 323.56] [color={rgb, 255:red, 0; green, 0; blue, 0 }  ][fill={rgb, 255:red, 0; green, 0; blue, 0 }  ][line width=0.75]      (0, 0) circle [x radius= 1.34, y radius= 1.34]   ;
\draw    (356,159.2) .. controls (369.8,167.2) and (373.4,134.4) .. (383,138.4) .. controls (392.6,142.4) and (389.4,173.2) .. (402.6,158.8) ;
\draw [shift={(402.6,158.8)}, rotate = 312.51] [color={rgb, 255:red, 0; green, 0; blue, 0 }  ][fill={rgb, 255:red, 0; green, 0; blue, 0 }  ][line width=0.75]      (0, 0) circle [x radius= 1.34, y radius= 1.34]   ;
\draw [shift={(356,159.2)}, rotate = 30.1] [color={rgb, 255:red, 0; green, 0; blue, 0 }  ][fill={rgb, 255:red, 0; green, 0; blue, 0 }  ][line width=0.75]      (0, 0) circle [x radius= 1.34, y radius= 1.34]   ;
\draw  [dash pattern={on 0.84pt off 2.51pt}] (239.57,148.4) .. controls (239.57,102.9) and (276.46,66.01) .. (321.96,66.01) .. controls (367.46,66.01) and (404.35,102.9) .. (404.35,148.4) .. controls (404.35,193.9) and (367.46,230.79) .. (321.96,230.79) .. controls (276.46,230.79) and (239.57,193.9) .. (239.57,148.4) -- cycle ;
\draw    (303.8,219.4) -- (321.8,219.4) ;
\draw [shift={(300.8,219.4)}, rotate = 0] [fill={rgb, 255:red, 0; green, 0; blue, 0 }  ][line width=0.08]  [draw opacity=0] (3.57,-1.72) -- (0,0) -- (3.57,1.72) -- cycle    ;
\draw    (321.56,148.4) -- (268.25,205.61) ;
\draw [shift={(266.2,207.8)}, rotate = 312.98] [fill={rgb, 255:red, 0; green, 0; blue, 0 }  ][line width=0.08]  [draw opacity=0] (3.57,-1.72) -- (0,0) -- (3.57,1.72) -- cycle    ;
\draw    (292.87,133.58) -- (321.56,148.4) ;
\draw [shift={(290.2,132.2)}, rotate = 27.32] [fill={rgb, 255:red, 0; green, 0; blue, 0 }  ][line width=0.08]  [draw opacity=0] (3.57,-1.72) -- (0,0) -- (3.57,1.72) -- cycle    ;

\draw (328,23.6) node [anchor=north west][inner sep=0.75pt]    {$\rho $};
\draw (328.56,148.8) node [anchor=north west][inner sep=0.75pt]  [font=\footnotesize]  {$x_{0}$};
\draw (266,96.2) node [anchor=north west][inner sep=0.75pt]    {$U_{1}$};
\draw (355.8,96) node [anchor=north west][inner sep=0.75pt]    {$U_{2}$};
\draw (250.8,141.6) node [anchor=north west][inner sep=0.75pt]  [font=\footnotesize]  {$p_{1}$};
\draw (372.4,148.4) node [anchor=north west][inner sep=0.75pt]  [font=\footnotesize]  {$p_{2}$};
\draw (306.4,208.4) node [anchor=north west][inner sep=0.75pt]  [font=\tiny]  {$K$};
\draw (262.4,192.2) node [anchor=north west][inner sep=0.75pt]  [font=\tiny]  {$R$};
\draw (300.6,143.6) node [anchor=north west][inner sep=0.75pt]  [font=\tiny]  {$r$};

\end{tikzpicture}

    \caption{Local coarse separation.}
    \label{fig:lcs}
\end{figure}

See Figure~\ref{fig:lcs} for a cartoon of this definition. 

We now prove some basic lemmas about local coarse separation. 
We begin with the following lemma which says that we can decrease the value of $R$ in the above definition without destroying the local coarse separation property. 

\begin{lemma}\label{lem:shrink-R}
    Let $R >  r > K > 0$ and $L \geq 0$. Fix $x_0 \in X$. Let $\rho$ be a geodesic segment which $(R,r, L, K)$-locally coarsely separates at $x_0$. Let $R'$ be such that $r < R' \leq R$. 
    Then $\rho$ also $(R',r, L, K)$-locally coarsely separates at $x_0$. 
\end{lemma}

\begin{proof}
    Let $U_1, U_2 \subset N_R(x_0) \setminus N_K(\rho)$ be connected components, such that $U_i$ contains a path $p_i$ connecting $S_r(x_0)$ to $S_R(x_0)$ which is $L$-coarsely perpendicular to $\rho$ at $x_0$.  
    Now, it follows easily from the definition that a subpath of an $L$-coarsely perpendicular path is also $L$-coarsely perpendicular. Note that 
    $$U_1' := U_1 \cap N_{R'}(x_0), \ \ U_2' := U_2 \cap N_{R'}(x_0)$$ do not intersect a common connected component of $N_{R'}(x_0) \setminus N_K(\rho)$. Also, each $U_i'$ contains a subpath $p_i'$ of $p_i$ which connects $S_r(x_0)$ to $S_{R'}(x_0)$. As noted earlier, this subpath is $L$-coarsely perpendicular to $\rho$ at $x_0$. The lemma follows. 
\end{proof}

The following two lemmas are easy exercises in the triangle inequality, but helpful to note.

\begin{lemma}\label{lem:complement-agrees}
    Let $\rho \subset X$ be a geodesic, let $x_0 \in \rho$ be a vertex which lies at distance at least $5R$ from the endpoints of $\rho$, and let $K \geq 0$. 
    Let $\rho' \subset X$ be another geodesic such that $\rho \subset \rho'$. Then 
    $
    N_R(x_0) \setminus N_K(\rho) = N_R(x_0) \setminus N_K(\rho'). 
    $
\end{lemma}

\begin{lemma}\label{lem:trim-grow}
    Let $R >  r > K > 0$ and $L \geq 0$. Fix $x_0 \in X$. Let $\rho$ be a geodesic segment such that $x_0 \in \rho$, and $x_0$ lies at distance at least $5R$ from the endpoints of $\rho$. Let $\rho'$ be another geodesic segment such that $\rho \subset \rho'$. 
     Then $\rho$ is $(R,r, L, K)$-locally coarsely separating at $x_0$ if and only if $\rho'$ is. 
\end{lemma}

The latter of the previous two lemmas essentially says that if $\rho$ locally coarsely separating then we may `extend' or `trim' $\rho$ while preserving this property. 

We now state the following standard definition.

\begin{definition}[(Global) coarse separation]
    Given $K \geq 0$, a subset $A \subset X$ is said to \textit{$K$-coarsely separate $X$ into deep components} if $X \setminus N_K(A)$ contains two distinct connected components $U_1$, $U_2$ and neither of the $U_i$ is contained in $N_{K'}(A)$ for any $K' \geq 0$.

\end{definition}

We may omit mention of the constant $K$ and simply say that the subset \textit{$A$ coarsely separates $X$}.

\begin{lemma}\label{lem:coarsely-sep-geodesic}
    Let $r > K > 0$ and $L > 0$. Suppose that for all $R > r$ we have that there is a geodesic segment $\rho$ which $(R,r, L, K)$-locally coarsely separates. 
    Then $X$ contains a bi-infinite geodesic $\gamma$ which coarsely separates $X$. 
\end{lemma}

\begin{proof}
    Fix $x_0 \in X$. Since $X$ is a transitive graph, the hypotheses imply that there exists a sequence $(\rho_k)_{k \geq 1}$ of geodesic segments and a sequence of integers $(R_k)_{k\geq 1}$ with the following properties:
    \begin{enumerate}
        \item we have that $R_k \to \infty$ as $k \to \infty$, 
    
        \item for all $k \geq 1$, we have that $x_0 \in \rho_k$ and lies at distance at least $5R_k$ from the endpoints of $\rho_k$, and  

        \item the segment $\rho_k$ is $(R_k,r, L, K)$-locally  coarsely separating at $x_0$. 
    \end{enumerate}
    Contrary to our usual convention in this paper, we will parametrise each $\rho_k$ at unit speed on a \emph{symmetric} interval, so $\rho_k : [-\tfrac {\ell_k} 2, \tfrac {\ell_k} 2] \to X$ where $\ell_k = \length (\rho_k)$. We also assume that  $\rho_k(0) = x_0$ for all $k > 0$. 
    
    By the Arzela--Ascoli theorem, the sequence $(\rho_k)$ contains a subsequence $(\rho_{n_k})$ which converges uniformly on compact subsets to a bi-infinite geodesic $\gamma$. 
    In particular, for all $N \geq 1$ there exists $M \geq 1$ such that for all $k \geq M$ we have 
    $$
    \gamma|_{[-N,N]} = \rho_{n_k}|_{[-N,N]}.
    $$ 
    We will abuse notation and simply write $\rho_k$ and $R_k$ for $\rho_{n_k}$ and $R_{n_k}$. 
    
    By Lemmas~\ref{lem:shrink-R} and \ref{lem:trim-grow} we may again pass to a subsequence, decrease the values of $R_k$, and `trim' our geodesics so that we may assume without loss of generality that 
    $$
    \rho_{1} \subset \rho_{2} \subset \rho_3 \subset \ldots \subset  \bigcup_{k \geq 1} \rho_k = \gamma. 
    $$
    By Lemma~\ref{lem:complement-agrees}, we have that 
    $$
    N_{R_k}(x_0) \setminus N_K(\rho_k) = N_{R_k}(x_0) \setminus N_K(\gamma). 
    $$
    For each $k \geq 1$, let $U_1^{(k)}$, $U_2^{(k)}$ denote two distinct connected components of $N_{R_k}(x_0) \setminus N_K(\rho_k)$, and for each $i =1,2$ let $p_i^{(k)} \subset U_i^{(k)}$ be a path which connects $N_r(x_0)$ to $S_{R_k}(x_0)$, and diverges $L$-linearly from $\rho_k$. By Lemma~\ref{lem:trim-grow}, $p_i^{(k)}$ also $L$-linearly diverges from $\gamma$. 
    In particular, each $U_i^{(k)}$ contains vertices which are arbitrarily far from $\gamma$ as $k \to \infty$.

    We would now like to say that $U_i^{(k)} \subset U_i^{(m)}$ for all $k < m$, $i = 1,2$. As it stands, this is not currently true as there are potentially many choices for the $U_i^{(k)}$. In order to ensure this, we will need to modify our choices of the $U_i^{(k)}$. Given $m \geq k$, let $q_i^{(k,m)}$ be a subpath of $p_i^{(m)}$ contained in $N_{R_k}(x_0) \setminus N_K(\gamma)$ connecting $N_r(x_0)$ to $S_{R_k}(x_0)$, and let 
    $W_i^{(k,m)}$  be the connected component of $ N_{R_k}(x_0) \setminus N_K(\gamma)$ which contains $q_i^{(k,m)}$. 
    It is clear that the $q_i^{(k,m)}$ and $W_i^{(k,m)}$ also witness the fact that $\rho_k$ is $(R_k,r,L,K)$-locally coarsely separating at $x_0$, for any $m \geq k$. In particular, the previous claims about the $U_i^{(k)}$ also apply to the $W_i^{(k,m)}$. Note that 
    $$
    W_i^{(k,m)} \subset W_i^{(k',m)}
    $$
    for all $k \leq k' \leq m$, by construction. 
    
    Now, if we fix $k$ and vary $m > k$, then $W_i^{(k,m)}$ can only take a bounded number of possibilities, as every $W_i^{(k,m)}$ is a connected component of $N_{R_k}(x_0) \setminus N_K(\gamma)$ which intersects $N_r(x_0)$. In particular, there are only at most $|N_r(x_0)|$ many possibilities. Fix $k = 1$, $i =1$, and consider the sequence 
    $$
    W_1^{(1,2)}, W_1^{(1,3)},  W_1^{(1,4)}, \ldots. 
    $$
    By passing to a subsequence and relabelling our indices, we may assume that 
    $$
    W_1^{(1,2)} = W_1^{(1,3)}  = W_1^{(1,4)} =  \ldots. 
    $$
    Similarly, now considering $i = 2$ we may also assume that 
    $$
    W_2^{(1,2)} = W_2^{(1,3)}  = W_2^{(1,4)} =  \ldots
    $$
    by passing to another subsequence and again relabelling. We now define $V_i^{(1)} := W_i^{(1,2)}$ as above. Note that the $V_i^{(1)}$ still witness the fact that $\rho_1$ is $(R_1, r, L, K)$-locally coarsely separating.

    We now choose $V_i^{(k)}$ for $k > 1$. Here, the choice is basically made for us. Indeed, let $k > 1$ and $i \in \{1,2\}$ and consider the sequence 
    $
    W_i^{(k,k)}, W_i^{(k,k+1)},  W_i^{(k,k+2)}, \ldots. 
    $ 
    This sequence is now already constant, as by construction we have that each $W_i^{(k,m)}$ is a connected component of $N_{R_k}(x_0) \setminus N_K(\gamma)$. Moreover, we have for any $m \geq k$ that
    $$
    V_i^{(1)} = W_i^{(1,2)} = W_i^{(1,m)} \subset W_i^{(k,m)},
    $$
    for $i = 1,2$. Since $V_i^{(1)}$ is connected subgraph of $N_{R_k}(x_0) \setminus N_K(\gamma)$, this completely determines $W_i^{(k,m)}$. We thus define  $V_i^{(k)} = W_i^{(k,k)}$ inductively for all $k \in \N$, $i = 1,2$. 
    
    By replacing the $U_i^{(k)}$ with the newly chosen $V_i^{(k)}$, we have the new property that 
    $$
    U_i^{(1)} \subset U_i^{(2)} \subset U_i^{(3)} \subset \ldots \subset \bigcup_{n \geq 1} U_i^{(n)} =: U_i  
    $$
    for each $i = 1,2$. Now, we certainly have that $U_1$ and $U_2$ are disjoint and connected. Moreover, each $U_i$ contains points arbitrarily far away from $\gamma$. We claim that there is no path $p$ from $U_1$ to $U_2$ avoiding $N_K(\gamma)$. Indeed, if one existed then there would exist $R \geq 0$ such that  $p \subset N_R(x_0) \setminus N_K(\gamma)$. Choose $k \geq 1$ such that $R_k \geq R$. Then we have that $U_1^{(k)}$ and $U_2^{(k)}$ are in the same connected component of $N_R(x_0) \setminus N_K(\rho_k)$. This is a contradiction. It follows that $\gamma$ coarsely separates $X$. 
\end{proof}

We will need the following theorem of Papasoglu. In what follows, a \textit{quasi-line} in $X$ is the image of a uniformly proper map $\R \to X$. See \cite{papasoglu2005quasi} for a precise definition. Examples of uniformly proper maps include quasi-isometries and coarse embeddings, and this is sufficient for our purposes. 

\begin{theorem}[Papasoglu \cite{papasoglu2005quasi}]\label{thm:papasoglu}
    Suppose that there is some quasi-line $\gamma$ which coarsely separates $X$.  Then $G$ splits non-trivially over a two-ended subgroup. 
\end{theorem}

In all that follows, \textbf{we now additionally assume that $G$ does not split over a two-ended subgroup}. 
In particular, the above theorem tells us that no bi-infinite geodesic coarsely separates $X$ into multiple deep components. 
Combining Theorem~\ref{thm:papasoglu} with Lemma~\ref{lem:coarsely-sep-geodesic}, we get the following corollary. 

\begin{corollary}
    For all $r > K > 0$, $L > 0$, there exists $R > 0$ such that there is no geodesic segment in $X$ which is $(R, r, L , K)$-locally coarsely separating. 
\end{corollary}

By simply rephrasing this corollary in more detail, we arrive at what we call the `diversion lemma'. 

\begin{lemma}[Diversion lemma]\label{lem:diversion}
    Given $K, L > 0$ there exists $R = R(L , K) > K$ such that the following holds: 
    
    Let $\rho \subset X$ be a geodesic, and fix $x_0 \in \rho$ which is at distance at least $5R$ from the endpoints of $\rho$. Let $q_1, q_2 \subset N_R(x_0)$ be  paths which are $L$-coarsely perpendicular to $\rho$, originating in $S_R(x_0)$ and terminating in $N_K(\rho)$. Then there exists a path from a vertex in $q_1$ to a vertex in $q_2$ contained entirely in $N_R(x_0) \setminus N_K(\rho)$. 
\end{lemma}

\begin{remark}
	Note that a similar statement does not hold for arbitrary finitely generated groups. Indeed, the lamplighter group contains a separating quasi-line but does not split over a two-ended subgroup \cite{papasoglu2012splittings}. 
\end{remark}

\subsection{Constructing minors in groups which don't split}

We now begin constructing fat minors in $X$. 
First, note the following standard result about vertex transitive graphs.

\begin{proposition}
    Let $X$ be a infinite, locally finite, vertex-transitive graph. Then $X$ contains a bi-infinite geodesic. 
\end{proposition}

\begin{proof}
	This follows easily from the Arzela--Ascoli theorem, together the fact that $X$ contains arbitrarily long geodesic segments whose midpoints all coincide thanks to vertex-transitivity. 
\end{proof}

For the remainder of this section, let $\gamma \subset X$ denote some fixed bi-infinite geodesic.

\begin{lemma}[Constant height paths]\label{lem:constant-height-paths}
There exists $\varepsilon_0 > 0$ depending only on $X$ such that the following holds:   

Let $r \geq 0$. Let $a, b \in X$ be such that 
$
\dist(a, \gamma) = \dist(b, \gamma) = r 
$. 
Assume further that there exists a path connecting $a$ to $b$ disjoint from $N_{r - 1}(\gamma)$. Then there exists a path $p$ from $a$ to $b$ contained in $A_{r\pm\varepsilon_0}(\gamma)$.
\end{lemma}

\begin{proof}
    This is an easy consequence of $G$ being finitely presented. Take $\varepsilon_0 > 0$ to be the length of the longest defining relator in a presentation corresponding to $X$.
\end{proof}

Recall the definition of a geodesic which is normal to a given subgraph (Definition~\ref{def:normal}). From now on, we will refer to a geodesic $\rho$ which is normal to $\gamma$ as simply a \textit{normal}. 
The following lemma is an easy consequence of the fact that $G$ is one-ended.

\begin{lemma}\label{lem:normals-exist}
    Fix $m > r \geq 0$.  
    Let $U \subset X \setminus N_r(\gamma)$ denote the unique deep component. Then there exists infinitely many distinct normals of length $m$ which terminate in $U$. 
\end{lemma}

In particular, we can find normals which are arbitrarily far apart from each other. 
Next, we note that almost constant-height paths are coarsely perpendicular to normals, with very controlled constants.

\begin{lemma}\label{lem:constant-height-coarse-perp}
    Let $r, \varepsilon > 0$. Let $\rho \subset X$ be a normal to $\gamma$ and let $x_0 \in \rho$ be such that $\dist(x_0, \gamma) = r$. Given $q \subset A_{r\pm \varepsilon}(\gamma)$, we have that $q$ is $L$-coarsely perpendicular to $\rho$ at $x_0$, where $L = \max\{2,\varepsilon\}$. 
\end{lemma}

\begin{proof}
    This follows immediately from the triangle inequality. 
\end{proof}

We now proceed with our general construction.

\begin{lemma}\label{lem:no-split}
    Let $G$ be a one-ended finitely presented group which is not a virtual surface group, and suppose that $G$ does not split over a two-ended subgroup. Then $G$ is not asymptotically minor-excluded. 
\end{lemma}

\begin{proof}
    Fix integers, $n, K > 0$.
    Fix $\varepsilon_0 > 0$ as in Lemma~\ref{lem:constant-height-paths}. Let $L = \max\{2, \varepsilon_0\}$. Choose $R = R(L , K)> 0$ as in the Diversion Lemma (\ref{lem:diversion}). Let $\gamma \subset X$ be some choice of bi-infinite geodesic. 

    Write $m = \tbinom n 2$. Let $h = 3mR$. Given $i \in \{1, \ldots, m\}$, let 
    $
    h_i = (3i - \tfrac3 2)R
    $. 
    Note that $R < h_i < h-R$ for all $i$, and $|h_i - h_j| > 2R$ for all $i \neq j$. 
    Let $U$ denote the unique deep component of $X \setminus N_h(\gamma)$. By Lemma~\ref{lem:normals-exist}, let $\rho_1, \ldots, \rho_n$ be a collection of normals of $\gamma$ of length $h$, all terminating in $U$. Since we can choose the $\rho_i$ to be arbitrarily far apart, we may ensure that 
    $
    N_{2R}(\rho_i) \cap N_{2R}(\rho_j) = \emptyset
    $ 
    for all $i \neq j$. Without loss of generality, assume that $\rho_i(0) \in \gamma$ for all $i$. 
    
    Choose some ordering of the set of pairs 
    $$
    S = \{(i, j ) : 1\leq i < j \leq m\},
    $$
    say $S = \{P_1, \ldots, P_m\}$, where $P_i = (x_i, y_i)$. 
    
    We will now attempt to construct a $K_n$-minor na\"ively, before modifying it into a $K$-fat minor. For each $i \in \{1, \ldots, m\}$, apply Lemma~\ref{lem:constant-height-paths} and let $q_i$ be a path connecting $\rho_{x_i}(h_i)$ to $\rho_{y_i}(h_i)$ such that 
    $$
    q_i \subset A_{h_i\pm\varepsilon_0}(\gamma). 
    $$
    We will modify each $q_i$ such that it avoids the $K$-neighbourhoods of all $\rho_j$ where $j \not\in\{x_i,y_i\}$. Once this is achieved, the union of the $\rho_i$ and the $q_j$ will form a $K$-fat $K_n$-minor, where the $\rho_i$ are the branch sets and the $q_j$ are the edge paths. 

    To ease notation, let us fix $i$ and write $q = q_i$. Write $\rho_+ = \rho_{x_i}$, and $\rho_- = \rho_{y_i}$.
    Suppose there exists some $\rho_j$ distinct from $\rho_+$, $\rho_-$ such that $q$ intersects $N_K(\rho_j)$. Write $x_0 = \rho_j(h_i)$. Parameterise $q$ such that it begins at $\rho_+$ and terminates at $\rho_-$. Let $q_+ \subset q$ denote the very first contiguous subpath of $q$ which is contained in $N_R(x_0)$, begins in $S_R(x_0)$ and terminates in $N_K(\rho_j)$. Similarly, let $q_- \subset q$ be the very last contiguous subpath of $q$ which is contained in $N_R(x_0)$, begins in $N_K(\rho_j)$ and terminates in $S_R(x_0)$. Let $a \in S_R(x_0)$ be the initial vertex of $q_+$ and $b \in S_R(x_0)$ be the terminal vertex of $q_-$. By the diversion lemma (\ref{lem:diversion}), there is a path through $N_R(x_0)$ connecting $a$ to $b$ which avoids $N_K(\rho_j)$. We replace the midsection of $q$ with this new path. 

    If we repeat this process at every one of the `bad intersections', it is clear that the resulting figure is a $K$-fat $K_n$-minor. See Figure~\ref{fig:main-construction} for a cartoon of this construction. The proposition follows. 
\end{proof}

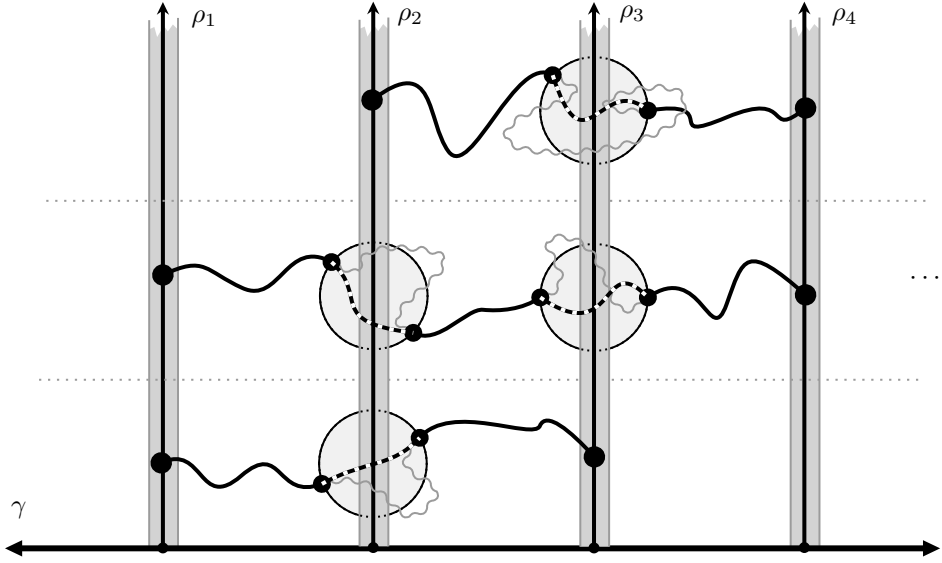
\begin{figure}
    \centering

\tikzset{every picture/.style={line width=0.75pt}} 

\begin{tikzpicture}[x=0.75pt,y=0.75pt,yscale=-1,xscale=1]

\draw  [draw opacity=0][fill={rgb, 255:red, 211; green, 211; blue, 211 }  ,fill opacity=1 ] (181.5,25) -- (184,21.5) -- (186.5,24) -- (186.56,284.2) -- (172.18,284.2) -- (172,24) -- (174.75,22.75) -- (176.13,27.82) -- (179.5,22.25) -- cycle ;
\draw [line width=1.5]    (179.02,14.2) -- (179.02,284.8) ;
\draw [shift={(179.02,284.8)}, rotate = 90] [color={rgb, 255:red, 0; green, 0; blue, 0 }  ][fill={rgb, 255:red, 0; green, 0; blue, 0 }  ][line width=1.5]      (0, 0) circle [x radius= 1.74, y radius= 1.74]   ;
\draw [shift={(179.02,10.2)}, rotate = 90] [fill={rgb, 255:red, 0; green, 0; blue, 0 }  ][line width=0.08]  [draw opacity=0] (6.43,-3.09) -- (0,0) -- (6.43,3.09) -- (4.27,0) -- cycle    ;
\draw [color={rgb, 255:red, 155; green, 155; blue, 155 }  ,draw opacity=1 ]   (172.18,19.33) -- (172.18,284.2) ;
\draw [color={rgb, 255:red, 155; green, 155; blue, 155 }  ,draw opacity=1 ]   (186.56,18.2) -- (186.56,284.2) ;
\draw  [draw opacity=0][fill={rgb, 255:red, 211; green, 211; blue, 211 }  ,fill opacity=1 ] (501.7,24.8) -- (504.2,21.3) -- (506.7,23.8) -- (506.76,284) -- (492.38,284) -- (492.2,23.8) -- (494.95,22.55) -- (496.33,27.62) -- (499.7,22.05) -- cycle ;
\draw [line width=1.5]    (499.22,14) -- (499.22,284.6) ;
\draw [shift={(499.22,284.6)}, rotate = 90] [color={rgb, 255:red, 0; green, 0; blue, 0 }  ][fill={rgb, 255:red, 0; green, 0; blue, 0 }  ][line width=1.5]      (0, 0) circle [x radius= 1.74, y radius= 1.74]   ;
\draw [shift={(499.22,10)}, rotate = 90] [fill={rgb, 255:red, 0; green, 0; blue, 0 }  ][line width=0.08]  [draw opacity=0] (6.43,-3.09) -- (0,0) -- (6.43,3.09) -- (4.27,0) -- cycle    ;
\draw [color={rgb, 255:red, 155; green, 155; blue, 155 }  ,draw opacity=1 ]   (492.38,19.13) -- (492.38,284) ;
\draw [color={rgb, 255:red, 155; green, 155; blue, 155 }  ,draw opacity=1 ]   (506.76,18) -- (506.76,284) ;
\draw [color={rgb, 255:red, 155; green, 155; blue, 155 }  ,draw opacity=1 ] [dash pattern={on 0.84pt off 2.51pt}]  (117.2,200.2) -- (557.2,200.2) ;
\draw [color={rgb, 255:red, 155; green, 155; blue, 155 }  ,draw opacity=1 ] [dash pattern={on 0.84pt off 2.51pt}]  (120.8,110.2) -- (560.8,110.2) ;
\draw  [fill={rgb, 255:red, 241; green, 241; blue, 241 }  ,fill opacity=1 ] (257.42,157.99) .. controls (257.42,143.16) and (269.45,131.13) .. (284.28,131.13) .. controls (299.11,131.13) and (311.13,143.16) .. (311.13,157.99) .. controls (311.13,172.82) and (299.11,184.84) .. (284.28,184.84) .. controls (269.45,184.84) and (257.42,172.82) .. (257.42,157.99) -- cycle ;
\draw  [fill={rgb, 255:red, 241; green, 241; blue, 241 }  ,fill opacity=1 ] (367.4,158.79) .. controls (367.4,143.96) and (379.42,131.93) .. (394.25,131.93) .. controls (409.08,131.93) and (421.11,143.96) .. (421.11,158.79) .. controls (421.11,173.62) and (409.08,185.64) .. (394.25,185.64) .. controls (379.42,185.64) and (367.4,173.62) .. (367.4,158.79) -- cycle ;
\draw [line width=2.25]    (105.2,285) -- (562,285) ;
\draw [shift={(567,285)}, rotate = 180] [fill={rgb, 255:red, 0; green, 0; blue, 0 }  ][line width=0.08]  [draw opacity=0] (8.57,-4.12) -- (0,0) -- (8.57,4.12) -- cycle    ;
\draw [shift={(100.2,285)}, rotate = 0] [fill={rgb, 255:red, 0; green, 0; blue, 0 }  ][line width=0.08]  [draw opacity=0] (8.57,-4.12) -- (0,0) -- (8.57,4.12) -- cycle    ;
\draw  [fill={rgb, 255:red, 241; green, 241; blue, 241 }  ,fill opacity=1 ] (257,242.39) .. controls (257,227.56) and (269.02,215.53) .. (283.85,215.53) .. controls (298.68,215.53) and (310.71,227.56) .. (310.71,242.39) .. controls (310.71,257.22) and (298.68,269.24) .. (283.85,269.24) .. controls (269.02,269.24) and (257,257.22) .. (257,242.39) -- cycle ;
\draw  [fill={rgb, 255:red, 241; green, 241; blue, 241 }  ,fill opacity=1 ] (367.4,64.79) .. controls (367.4,49.96) and (379.42,37.93) .. (394.25,37.93) .. controls (409.08,37.93) and (421.11,49.96) .. (421.11,64.79) .. controls (421.11,79.62) and (409.08,91.64) .. (394.25,91.64) .. controls (379.42,91.64) and (367.4,79.62) .. (367.4,64.79) -- cycle ;
\draw  [draw opacity=0][fill={rgb, 255:red, 211; green, 211; blue, 211 }  ,fill opacity=1 ] (286.5,24.8) -- (289,21.3) -- (291.5,23.8) -- (291.56,284) -- (277.18,284) -- (277,23.8) -- (279.75,22.55) -- (281.13,27.62) -- (284.5,22.05) -- cycle ;
\draw [line width=1.5]    (284.02,14) -- (284.02,284.6) ;
\draw [shift={(284.02,284.6)}, rotate = 90] [color={rgb, 255:red, 0; green, 0; blue, 0 }  ][fill={rgb, 255:red, 0; green, 0; blue, 0 }  ][line width=1.5]      (0, 0) circle [x radius= 1.74, y radius= 1.74]   ;
\draw [shift={(284.02,10)}, rotate = 90] [fill={rgb, 255:red, 0; green, 0; blue, 0 }  ][line width=0.08]  [draw opacity=0] (6.43,-3.09) -- (0,0) -- (6.43,3.09) -- (4.27,0) -- cycle    ;
\draw [color={rgb, 255:red, 155; green, 155; blue, 155 }  ,draw opacity=1 ]   (277.18,19.13) -- (277.18,284) ;
\draw [color={rgb, 255:red, 155; green, 155; blue, 155 }  ,draw opacity=1 ]   (291.56,18) -- (291.56,284) ;
\draw  [draw opacity=0][fill={rgb, 255:red, 211; green, 211; blue, 211 }  ,fill opacity=1 ] (396.7,25) -- (399.2,21.5) -- (401.7,24) -- (401.76,284.2) -- (387.38,284.2) -- (387.2,24) -- (389.95,22.75) -- (391.33,27.82) -- (394.7,22.25) -- cycle ;
\draw [line width=1.5]    (394.22,14.2) -- (394.22,284.8) ;
\draw [shift={(394.22,284.8)}, rotate = 90] [color={rgb, 255:red, 0; green, 0; blue, 0 }  ][fill={rgb, 255:red, 0; green, 0; blue, 0 }  ][line width=1.5]      (0, 0) circle [x radius= 1.74, y radius= 1.74]   ;
\draw [shift={(394.22,10.2)}, rotate = 90] [fill={rgb, 255:red, 0; green, 0; blue, 0 }  ][line width=0.08]  [draw opacity=0] (6.43,-3.09) -- (0,0) -- (6.43,3.09) -- (4.27,0) -- cycle    ;
\draw [color={rgb, 255:red, 155; green, 155; blue, 155 }  ,draw opacity=1 ]   (387.38,19.33) -- (387.38,284.2) ;
\draw [color={rgb, 255:red, 155; green, 155; blue, 155 }  ,draw opacity=1 ]   (401.76,18.2) -- (401.76,284.2) ;
\draw [line width=1.5]    (178.2,242) .. controls (199.6,231.8) and (196.8,267) .. (218,248.6) .. controls (239.2,230.2) and (235.2,261.4) .. (258.4,252.6) ;
\draw [shift={(178.2,242)}, rotate = 334.52] [color={rgb, 255:red, 0; green, 0; blue, 0 }  ][fill={rgb, 255:red, 0; green, 0; blue, 0 }  ][line width=1.5]      (0, 0) circle [x radius= 4.36, y radius= 4.36]   ;
\draw [line width=1.5]    (394.4,239) .. controls (382,225.4) and (370.4,215.8) .. (367.2,223.4) .. controls (364,231) and (322.8,211) .. (307.2,229.4) ;
\draw [shift={(394.4,239)}, rotate = 227.64] [color={rgb, 255:red, 0; green, 0; blue, 0 }  ][fill={rgb, 255:red, 0; green, 0; blue, 0 }  ][line width=1.5]      (0, 0) circle [x radius= 4.36, y radius= 4.36]   ;
\draw  [dash pattern={on 0.75pt off 1.5pt}] (257.42,157.99) .. controls (257.42,143.16) and (269.45,131.13) .. (284.28,131.13) .. controls (299.11,131.13) and (311.13,143.16) .. (311.13,157.99) .. controls (311.13,172.82) and (299.11,184.84) .. (284.28,184.84) .. controls (269.45,184.84) and (257.42,172.82) .. (257.42,157.99) -- cycle ;
\draw  [dash pattern={on 0.75pt off 1.5pt}] (367.4,158.79) .. controls (367.4,143.96) and (379.42,131.93) .. (394.25,131.93) .. controls (409.08,131.93) and (421.11,143.96) .. (421.11,158.79) .. controls (421.11,173.62) and (409.08,185.64) .. (394.25,185.64) .. controls (379.42,185.64) and (367.4,173.62) .. (367.4,158.79) -- cycle ;
\draw  [dash pattern={on 0.75pt off 1.5pt}] (257,242.39) .. controls (257,227.56) and (269.02,215.53) .. (283.85,215.53) .. controls (298.68,215.53) and (310.71,227.56) .. (310.71,242.39) .. controls (310.71,257.22) and (298.68,269.24) .. (283.85,269.24) .. controls (269.02,269.24) and (257,257.22) .. (257,242.39) -- cycle ;
\draw  [dash pattern={on 0.75pt off 1.5pt}] (367.4,64.79) .. controls (367.4,49.96) and (379.42,37.93) .. (394.25,37.93) .. controls (409.08,37.93) and (421.11,49.96) .. (421.11,64.79) .. controls (421.11,79.62) and (409.08,91.64) .. (394.25,91.64) .. controls (379.42,91.64) and (367.4,79.62) .. (367.4,64.79) -- cycle ;
\draw [line width=1.5]    (179.02,147.5) .. controls (200.42,137.3) and (200.04,145.6) .. (218.82,154.1) .. controls (237.6,162.6) and (240.4,128.6) .. (263.2,141) ;
\draw [shift={(179.02,147.5)}, rotate = 334.52] [color={rgb, 255:red, 0; green, 0; blue, 0 }  ][fill={rgb, 255:red, 0; green, 0; blue, 0 }  ][line width=1.5]      (0, 0) circle [x radius= 4.36, y radius= 4.36]   ;
\draw [line width=1.5]    (499.82,157.5) .. controls (469.6,130.6) and (468.8,138.6) .. (460.4,162.2) .. controls (452,185.8) and (448,139.8) .. (421.11,158.79) ;
\draw [shift={(499.82,157.5)}, rotate = 221.67] [color={rgb, 255:red, 0; green, 0; blue, 0 }  ][fill={rgb, 255:red, 0; green, 0; blue, 0 }  ][line width=1.5]      (0, 0) circle [x radius= 4.36, y radius= 4.36]   ;
\draw [line width=1.5]    (499.82,63.5) .. controls (478.8,80.6) and (494.8,53.8) .. (460.4,68.2) .. controls (426,82.6) and (460,58.6) .. (421.11,64.79) ;
\draw [shift={(499.82,63.5)}, rotate = 140.87] [color={rgb, 255:red, 0; green, 0; blue, 0 }  ][fill={rgb, 255:red, 0; green, 0; blue, 0 }  ][line width=1.5]      (0, 0) circle [x radius= 4.36, y radius= 4.36]   ;
\draw [line width=1.5]    (283.42,59.5) .. controls (315.6,35) and (310.4,70.2) .. (321.6,85.4) .. controls (332.8,100.6) and (355.2,35.4) .. (373.6,47) ;
\draw [shift={(283.42,59.5)}, rotate = 322.72] [color={rgb, 255:red, 0; green, 0; blue, 0 }  ][fill={rgb, 255:red, 0; green, 0; blue, 0 }  ][line width=1.5]      (0, 0) circle [x radius= 4.36, y radius= 4.36]   ;
\draw [line width=1.5]    (304,176.6) .. controls (316.4,183.4) and (330.4,164.2) .. (338.8,165) .. controls (347.2,165.8) and (358.4,165) .. (367.4,158.79) ;
\draw [color={rgb, 255:red, 155; green, 155; blue, 155 }  ,draw opacity=1 ]   (263.2,141) .. controls (265.62,140.45) and (267.13,141.28) .. (267.74,143.5) .. controls (268.41,145.71) and (269.89,146.41) .. (272.2,145.61) .. controls (273.95,144.4) and (275.41,144.28) .. (276.59,145.25) .. controls (278.58,143.98) and (279,142.36) .. (277.83,140.4) .. controls (277.5,137.85) and (278.61,136.6) .. (281.15,136.65) .. controls (283.21,137.49) and (284.79,136.87) .. (285.89,134.79) .. controls (287.26,132.82) and (288.89,132.57) .. (290.76,134.02) .. controls (292.36,135.58) and (293.96,135.5) .. (295.57,133.79) .. controls (297.48,132.15) and (299.23,132.19) .. (300.8,133.92) .. controls (302.42,135.69) and (304.21,135.83) .. (306.17,134.33) .. controls (307.65,132.82) and (309.22,133.04) .. (310.89,134.98) .. controls (312.07,136.94) and (313.64,137.36) .. (315.6,136.23) .. controls (317.84,135.48) and (319.11,136.42) .. (319.41,139.04) .. controls (317.98,140.29) and (317.75,141.86) .. (318.73,143.76) .. controls (319,146.23) and (317.97,147.51) .. (315.66,147.59) .. controls (313.32,147.56) and (312.07,148.76) .. (311.9,151.18) .. controls (311.81,153.49) and (310.6,154.56) .. (308.27,154.37) .. controls (305.94,154.2) and (304.61,155.36) .. (304.3,157.87) .. controls (304.29,160.16) and (303.16,161.23) .. (300.93,161.09) .. controls (298.54,161.22) and (297.43,162.52) .. (297.62,164.99) .. controls (298.3,167.14) and (297.68,168.67) .. (295.75,169.58) .. controls (294.66,171.99) and (295.36,173.38) .. (297.86,173.76) .. controls (300.13,173.21) and (301.71,174.01) .. (302.58,176.16) -- (304,176.6) ;
\draw [color={rgb, 255:red, 155; green, 155; blue, 155 }  ,draw opacity=1 ]   (367.4,158.79) .. controls (368.47,156.74) and (370.07,156.18) .. (372.18,157.13) .. controls (374.55,157.68) and (375.88,156.7) .. (376.19,154.2) .. controls (375.46,152.16) and (376.02,150.59) .. (377.85,149.5) .. controls (379.28,147.57) and (378.94,145.99) .. (376.82,144.77) .. controls (374.62,143.9) and (373.89,142.39) .. (374.64,140.22) .. controls (375.25,137.89) and (374.44,136.42) .. (372.21,135.81) .. controls (369.94,134.86) and (369.43,133.31) .. (370.67,131.15) .. controls (372.33,131.25) and (373.54,130.35) .. (374.29,128.46) .. controls (375.98,126.73) and (377.75,126.77) .. (379.6,128.56) .. controls (380.83,130.43) and (382.34,130.75) .. (384.14,129.52) .. controls (386.45,128.7) and (388,129.52) .. (388.8,131.98) .. controls (388.87,134.22) and (390.03,135.32) .. (392.28,135.29) .. controls (394.63,135.62) and (395.57,136.96) .. (395.08,139.29) .. controls (394.39,141.5) and (395.14,142.99) .. (397.35,143.74) .. controls (399.54,144.66) and (400.17,146.21) .. (399.23,148.38) .. controls (398.3,150.7) and (398.89,152.32) .. (401,153.25) .. controls (403.09,154.1) and (403.69,155.68) .. (402.82,157.99) .. controls (401.96,160.12) and (402.7,161.55) .. (405.04,162.29) .. controls (407.21,162.28) and (408.48,163.31) .. (408.85,165.38) .. controls (410.83,167.05) and (412.48,166.77) .. (413.8,164.54) .. controls (414.17,162.38) and (415.55,161.43) .. (417.94,161.7) -- (421.11,158.79) ;
\draw [color={rgb, 255:red, 155; green, 155; blue, 155 }  ,draw opacity=1 ]   (373.6,47) .. controls (375.94,46.25) and (377.46,46.98) .. (378.15,49.19) .. controls (378.75,51.42) and (380.22,52.3) .. (382.57,51.81) .. controls (385.04,51.72) and (386.05,52.92) .. (385.6,55.42) .. controls (383.81,55.53) and (382.68,56.67) .. (382.21,58.85) .. controls (381.28,61.02) and (379.76,61.58) .. (377.63,60.54) .. controls (375.41,59.49) and (373.83,59.97) .. (372.88,62) .. controls (371.68,64.12) and (370.07,64.63) .. (368.05,63.53) .. controls (365.96,62.5) and (364.34,63.1) .. (363.19,65.33) .. controls (362.43,67.48) and (361,68.15) .. (358.89,67.34) .. controls (356.48,66.83) and (355.06,67.77) .. (354.62,70.15) .. controls (354.66,72.39) and (353.53,73.63) .. (351.23,73.88) .. controls (348.88,74.93) and (348.49,76.49) .. (350.04,78.54) .. controls (352.05,78.72) and (353.17,79.91) .. (353.4,82.11) .. controls (354.3,84.34) and (355.8,84.92) .. (357.91,83.85) .. controls (359.98,82.63) and (361.64,82.98) .. (362.89,84.89) .. controls (364.4,86.78) and (366.16,86.98) .. (368.16,85.49) .. controls (369.77,83.92) and (371.31,84) .. (372.8,85.73) .. controls (374.41,87.43) and (376.05,87.44) .. (377.73,85.77) .. controls (379.38,84.06) and (381.1,84.01) .. (382.89,85.6) .. controls (384.77,87.15) and (386.54,87.03) .. (388.21,85.23) .. controls (389.38,83.44) and (390.96,83.28) .. (392.93,84.73) .. controls (394.96,86.14) and (396.54,85.93) .. (397.67,84.08) .. controls (399.18,82.15) and (400.98,81.84) .. (403.05,83.15) .. controls (404.72,84.51) and (406.26,84.18) .. (407.67,82.17) .. controls (408.98,80.14) and (410.68,79.7) .. (412.78,80.86) .. controls (414.92,81.97) and (416.55,81.46) .. (417.66,79.35) .. controls (418.63,77.24) and (420.15,76.67) .. (422.22,77.65) .. controls (424.34,78.55) and (425.73,77.92) .. (426.4,75.75) .. controls (427.16,73.45) and (428.69,72.57) .. (430.98,73.11) .. controls (433.33,73.46) and (434.58,72.48) .. (434.72,70.16) .. controls (434.61,67.83) and (435.66,66.52) .. (437.89,66.23) .. controls (440.12,65.36) and (440.64,63.81) .. (439.45,61.58) .. controls (437.76,60.29) and (437.39,58.72) .. (438.36,56.89) .. controls (438.08,54.38) and (436.66,53.44) .. (434.1,54.05) .. controls (432.26,55.3) and (430.64,54.99) .. (429.24,53.13) .. controls (427.68,51.39) and (426.06,51.37) .. (424.39,53.07) .. controls (422.85,54.86) and (421.13,55.03) .. (419.24,53.58) .. controls (417.21,52.23) and (415.63,52.52) .. (414.51,54.45) .. controls (413.27,56.46) and (411.65,56.89) .. (409.64,55.72) .. controls (407.35,54.73) and (405.71,55.32) .. (404.72,57.5) .. controls (404.16,59.61) and (402.76,60.39) .. (400.53,59.83) .. controls (397.95,60.78) and (397.99,62.02) .. (400.64,63.57) .. controls (402.64,62.42) and (404.29,62.79) .. (405.58,64.66) .. controls (407.15,66.47) and (408.76,66.57) .. (410.42,64.97) .. controls (412.18,63.32) and (414.05,63.32) .. (416.03,64.97) .. controls (417.72,66.59) and (419.41,66.53) .. (421.11,64.79) -- (421.11,64.79) ;
\draw [color={rgb, 255:red, 155; green, 155; blue, 155 }  ,draw opacity=1 ]   (258.4,252.6) .. controls (259.86,250.49) and (261.49,250.09) .. (263.28,251.4) .. controls (265.17,252.77) and (266.95,252.54) .. (268.6,250.7) .. controls (270.08,248.99) and (271.64,249.01) .. (273.28,250.75) .. controls (274.75,252.64) and (276.46,252.96) .. (278.41,251.72) .. controls (280.55,250.71) and (282.08,251.35) .. (283.01,253.64) .. controls (283.5,255.89) and (284.85,256.81) .. (287.08,256.39) .. controls (289.39,256.18) and (290.59,257.29) .. (290.68,259.71) .. controls (290.57,262.02) and (291.72,263.26) .. (294.11,263.41) .. controls (296.4,263.42) and (297.55,264.56) .. (297.55,266.84) .. controls (298.22,269.29) and (299.75,270) .. (302.14,268.98) .. controls (303.23,267.21) and (304.81,266.78) .. (306.86,267.67) .. controls (309.18,268.09) and (310.53,267.02) .. (310.92,264.47) .. controls (310.73,262.28) and (311.78,261.11) .. (314.06,260.94) .. controls (316.4,260.17) and (316.94,258.65) .. (315.68,256.38) .. controls (313.69,255.88) and (312.72,254.56) .. (312.76,252.41) .. controls (312.05,250.06) and (310.61,249.28) .. (308.43,250.05) .. controls (306.1,250.74) and (304.59,249.93) .. (303.9,247.61) .. controls (303.77,245.42) and (302.61,244.27) .. (300.42,244.15) .. controls (298.34,242.72) and (298.25,241.1) .. (300.14,239.3) .. controls (302.25,238.59) and (303.05,237.08) .. (302.53,234.79) .. controls (302.21,232.47) and (303.3,231.16) .. (305.79,230.86) -- (307.2,229.4) ;
\draw [color={rgb, 255:red, 0; green, 0; blue, 0 }  ,draw opacity=1 ][line width=1.5]    (263.2,141) .. controls (282.4,158.6) and (259.6,171) .. (304,176.6) ;
\draw [shift={(304,176.6)}, rotate = 7.19] [color={rgb, 255:red, 0; green, 0; blue, 0 }  ,draw opacity=1 ][fill={rgb, 255:red, 0; green, 0; blue, 0 }  ,fill opacity=1 ][line width=1.5]      (0, 0) circle [x radius= 3.48, y radius= 3.48]   ;
\draw [shift={(263.2,141)}, rotate = 42.51] [color={rgb, 255:red, 0; green, 0; blue, 0 }  ,draw opacity=1 ][fill={rgb, 255:red, 0; green, 0; blue, 0 }  ,fill opacity=1 ][line width=1.5]      (0, 0) circle [x radius= 3.48, y radius= 3.48]   ;
\draw [color={rgb, 255:red, 0; green, 0; blue, 0 }  ,draw opacity=1 ][line width=1.5]    (258.4,252.6) .. controls (280.4,240.2) and (292.8,245) .. (307.2,229.4) ;
\draw [shift={(307.2,229.4)}, rotate = 312.71] [color={rgb, 255:red, 0; green, 0; blue, 0 }  ,draw opacity=1 ][fill={rgb, 255:red, 0; green, 0; blue, 0 }  ,fill opacity=1 ][line width=1.5]      (0, 0) circle [x radius= 3.48, y radius= 3.48]   ;
\draw [shift={(258.4,252.6)}, rotate = 330.59] [color={rgb, 255:red, 0; green, 0; blue, 0 }  ,draw opacity=1 ][fill={rgb, 255:red, 0; green, 0; blue, 0 }  ,fill opacity=1 ][line width=1.5]      (0, 0) circle [x radius= 3.48, y radius= 3.48]   ;
\draw [color={rgb, 255:red, 0; green, 0; blue, 0 }  ,draw opacity=1 ][line width=1.5]    (367.4,158.79) .. controls (406,185.4) and (397.2,133.8) .. (421.11,158.79) ;
\draw [shift={(421.11,158.79)}, rotate = 46.27] [color={rgb, 255:red, 0; green, 0; blue, 0 }  ,draw opacity=1 ][fill={rgb, 255:red, 0; green, 0; blue, 0 }  ,fill opacity=1 ][line width=1.5]      (0, 0) circle [x radius= 3.48, y radius= 3.48]   ;
\draw [shift={(367.4,158.79)}, rotate = 34.59] [color={rgb, 255:red, 0; green, 0; blue, 0 }  ,draw opacity=1 ][fill={rgb, 255:red, 0; green, 0; blue, 0 }  ,fill opacity=1 ][line width=1.5]      (0, 0) circle [x radius= 3.48, y radius= 3.48]   ;
\draw [color={rgb, 255:red, 0; green, 0; blue, 0 }  ,draw opacity=1 ][line width=1.5]    (373.6,47) .. controls (389.6,97.4) and (398,44.6) .. (421.11,64.79) ;
\draw [shift={(421.11,64.79)}, rotate = 41.14] [color={rgb, 255:red, 0; green, 0; blue, 0 }  ,draw opacity=1 ][fill={rgb, 255:red, 0; green, 0; blue, 0 }  ,fill opacity=1 ][line width=1.5]      (0, 0) circle [x radius= 3.48, y radius= 3.48]   ;
\draw [shift={(373.6,47)}, rotate = 72.39] [color={rgb, 255:red, 0; green, 0; blue, 0 }  ,draw opacity=1 ][fill={rgb, 255:red, 0; green, 0; blue, 0 }  ,fill opacity=1 ][line width=1.5]      (0, 0) circle [x radius= 3.48, y radius= 3.48]   ;
\draw [color={rgb, 255:red, 255; green, 255; blue, 255 }  ,draw opacity=1 ][line width=1.5]  [dash pattern={on 1.69pt off 2.76pt}]  (263.2,141) .. controls (282.4,158.6) and (259.6,171) .. (304,176.6) ;
\draw [color={rgb, 255:red, 255; green, 255; blue, 255 }  ,draw opacity=1 ][line width=1.5]  [dash pattern={on 1.69pt off 2.76pt}]  (258.4,252.6) .. controls (280.4,240.2) and (292.8,245) .. (307.2,229.4) ;
\draw [color={rgb, 255:red, 255; green, 255; blue, 255 }  ,draw opacity=1 ][line width=1.5]  [dash pattern={on 1.69pt off 2.76pt}]  (367.4,158.79) .. controls (406,185.4) and (397.2,133.8) .. (421.11,158.79) ;
\draw [color={rgb, 255:red, 255; green, 255; blue, 255 }  ,draw opacity=1 ][line width=1.5]  [dash pattern={on 1.69pt off 2.76pt}]  (373.6,47) .. controls (389.6,97.4) and (398,44.6) .. (421.11,64.79) ;

\draw (101.8,259.8) node [anchor=north west][inner sep=0.75pt]    {$\gamma $};
\draw (191.8,13) node [anchor=north west][inner sep=0.75pt]    {$\rho _{1}$};
\draw (294.6,12.2) node [anchor=north west][inner sep=0.75pt]    {$\rho _{2}$};
\draw (405.4,10.6) node [anchor=north west][inner sep=0.75pt]    {$\rho _{3}$};
\draw (511.4,11.8) node [anchor=north west][inner sep=0.75pt]    {$\rho _{4}$};
\draw (550.2,144.6) node [anchor=north west][inner sep=0.75pt]    {$\cdots $};

\end{tikzpicture}

    \caption{Constructing a fat $K_n$ in a group which doesn't split. Using the diversion lemma (\ref{lem:diversion}), the `wiggly' paths which break the fatness of the minor are replaced with the striped paths which avoid the $K$-neighbourhoods of the $\rho_i$.}
    \label{fig:main-construction}
\end{figure}

Finally, we can deduce our main result. 

\fp*

\begin{proof}
    If a finite index subgroup of $G$ admits a planar Cayley graph, then clearly $G$ is quasi-isometric to a planar graph. Since asymptotic minors are invariant under quasi-isometry, we conclude that $G$ is asymptotically minor-excluded. 

    Now, suppose that no finite index subgroup of $G$ admits a planar Cayley graph. 
    Let $X$ be a Cayley graph of $G$. 
    We first assume that $G$ is one-ended, and so $G$ is not a virtual surface group. If $G$ does not split over a two-ended subgroup, then by Lemma~\ref{lem:no-split} we are done. Thus, let us assume that $G$ splits as an amalgam $G \cong A \ast_C B$ or an HNN extension $G \cong A \ast _C$ where $C$ is two-ended. For the sake of clarity, we will assume that $G$ is an HNN extension; the amalgam case is identical. 

    First, note that $A$ is itself necessarily finitely presented \cite{haglund2021note}. By Dunwoody's accessibility theorem \cite{dunwoody1985accessibility}, either $A$ is virtually free or $A$ contains a one-ended, finitely presented subgroup $H$. If $H$ is a virtual surface group then we are done by Theorem~\ref{thm:surface-subgroup}. Otherwise, we may repeat the preceding argument with $H$ in place of $G$. We will return to this inductive step shortly.

    Suppose now that $A$ is virtually free. By a corollary of the combination theorem of Bestvina--Feighn \cite{bestvina1992combination, bestvina1996addendum}, if $G$ is not hyperbolic then $G$ contains some Baumslag--Solitar subgroup. Since $G$ is not virtually $\Z^2$, it follows from Proposition~\ref{lem:bs} and Theorem~\ref{thm:surface-subgroup} that $G$ is not asymptotically minor-excluded. Thus we assume that $G$ is hyperbolic. By a theorem of Wise \cite[Thm.~4.19]{wise2000subgroup} we have that $G$ is residually finite and virtually torsion-free, and thus $G$ contains a finite index subgroup $K$ isomorphic to a one-ended, hyperbolic graph of free groups with cyclic edge groups. By a theorem of Wilton \cite{wilton2012one}, we have that $K$ contains a one-ended subgroup $H$ of  infinite index. A theorem of Karrass and Solitar \cite{karrass1970subgroups} implies that $K$ is coherent\footnote{Recall that a finitely presented group is said to be \textit{coherent} if every finitely generated subgroup is finitely presented.} and so $H$ is finitely presented. If $H$ a virtual surface group, then we are once again done by Theorem~\ref{thm:surface-subgroup}, otherwise we repeat this argument with $H$ in place of $G$. 

    In each of the above cases considered, we have either concluded that $G$ is not asymptotically minor-excluded, or we have found an infinite-index, one-ended, finitely presented subgroup $H$ which is not a virtual surface group. We now repeat this process by iteratively passing to this subgroup at each stage. If this process never terminates then we find a descending chain 
    $$
    G \geq H_1 \geq H_{2} \geq \ldots \geq H_i \geq \ldots
    $$
    of one-ended subgroups, where $|H_i : H_{i+1}| = \infty$ for every $i \geq 1$. By Theorem~\ref{thm:descending-chain} we deduce that $G$ is not asymptotically minor-excluded. 

    Now, let us assume that $G$ is infinite-ended and does not virtually admit a planar Cayley graph.  
    By Dunwoody's accessibility theorem \cite{dunwoody1985accessibility}, $G$ splits as a graph of groups with finite edge groups and vertex groups with at most one end. Let $H_1, \ldots, H_n$ be the one-ended vertex groups of this splitting. There must be at least one such vertex group lest $G$ is virtually free. Since $G$ is not virtually free we have that $n \geq 1$. 
    If each $H_i$ contains a finite index subgroup $K_i$ with a planar Cayley graph, then each $H_i$ is residually finite and virtually torsion-free. It is an easy exercise to deduce that $G$ itself is also virtually torsion-free and so $G$ is virtually a free product of free and surface groups. It follows that $G$ virtually admits a planar Cayley graph, a contradiction. Thus, some $H_i$ must not virtually admit a planar Cayley graph, and the result follows from the one-ended case discussed above. 
\end{proof}

\subsection*{Acknowledgements}

I am grateful to Panos Papasoglu for many stimulating conversations on this topic, and to Misha Schmalian for highlighting the simple idea behind the construction in Figure~\ref{fig:circles}. Finally, I thank the referees for the careful reading and detailed feedback. This work was supported by the Heilbronn Institute for Mathematical Research.

\bibliographystyle{abbrv}
\bibliography{references}

\end{document}